\documentclass[a4paper, 11pt]{preprint}

\usepackage[T1]{fontenc} 
\usepackage[utf8]{inputenc} 

\usepackage[english]{babel} 

\usepackage{amsthm, amsmath, amssymb, mathtools} 

\usepackage[osf]{newtxtext} 

\usepackage{mathrsfs} 

\usepackage{dsfont} 

\renewcommand{\hat}{\widehat} 

\usepackage{microtype} 

\usepackage{mhequ} 
\usepackage{comment} 

\usepackage{tikz}
\usetikzlibrary{decorations.markings}
\usetikzlibrary{arrows.meta}
\usetikzlibrary{decorations.pathmorphing}
\tikzset{snake/.style={coil,aspect=0}}
\usetikzlibrary{shapes.misc}

\usepackage[colorlinks=true, linkcolor=colorLink, citecolor=colorCite, urlcolor=colorLink, linktocpage=true]{hyperref}

\usepackage{orcidlink} 

\usepackage{enumitem}

\newcommand{\R}{\mathbb{R}} 
\newcommand{\N}{\mathbb{N}} 
\newcommand{\T}{\mathbb{T}} 
\newcommand{\E}{\mathbb{E}} 
\newcommand{\PP}{\mathbb{P}} 

\newcommand{\eps}{\varepsilon}
\newcommand{\D}{\mathrm{D}} 
\newcommand{\diff}[1]{\operatorname{Diff}^{#1}(\mathbb{T}^{d})}
\newcommand{\diffvol}[1]{\operatorname{Diff}_{\textnormal{Vol}}^{#1}(\mathbb{T}^{d})}
\newcommand{\homeo}{\operatorname{Homeo}(\mathbb{T}^{d})}
\newcommand{\flow}{\varphi}
\newcommand{\RR}{\mathbb{R}}
\newcommand{\TT}{\mathbb{T}}
\newcommand{\NN}{\mathbb{N}}
\newcommand{\ZZ}{\mathbb{Z}}
\newcommand{\EE}{\mathbb{E}}
\newcommand{\defeq}{\coloneqq}

\newcommand{\embed}{\hookrightarrow}
\newcommand{\dd}{\mathrm{d}\mkern0.5mu}
\newcommand{\supp}{\operatorname{supp}}
\newcommand{\cl}{\operatorname{cl}}
\newcommand{\vdiv}{\nabla\cdot}
\newcommand{\euler}{\mathrm{e}\mkern0.7mu}
\newcommand{\upi}{\mathrm{i}\mkern0.7mu}
\newcommand{\uppi}{\pi}
\newcommand{\place}{\,\cdot\,}
\newcommand{\ip}[1]{\langle#1\rangle}
\newcommand{\abs}[1]{\lvert#1\rvert}
\newcommand{\norm}[1]{\lVert#1\rVert}
\newcommand{\from}{\colon}
\newcommand{\meas}{\mu}

\newcommand{\str}{K}
\newcommand{\rds}{P}

\newcommand{\tzero}{|_{t=0}}
\newcommand{\cm}{H_{\textnormal{CM}}}
\newcommand{\trace}{\operatorname{tr}}
\newcommand{\trans}{\mathrm{T}}
\renewcommand{\complement}{\textnormal{c}}

\newcommand{\colour}{\theta} 
\newcommand{\shift}{\vartheta}
\newcommand{\diam}{\operatorname{diam}}
\newcommand{\law}{\operatorname{Law}}
\newcommand{\unif}{\boldsymbol{1}}
\newcommand{\reg}{\beta} 
\newcommand{\lin}{f} 
\newcommand{\scale}{\alpha} 
\newcommand{\diffu}{\kappa} 
\newcommand{\crit}{\textnormal{c}} 

\newcommand{\multiquad}[1][1]{\hspace*{#1em}\ignorespaces}

\newcommand{\init}[1]{
	\ifnumequal{#1}{0}{\bar{L}^{2}(\TT^{d})}{}
	\ifnumequal{#1}{1}{\bar{H}^{1}(\TT^{d})}{}
}%

\definecolor{colorLink}{RGB}{0,100,162} 
\definecolor{colorCite}{RGB}{8,124,100} 

\theoremstyle{plain}
\newtheorem{theorem}{Theorem}[section]
\newtheorem{corollary}[theorem]{Corollary}
\newtheorem{lemma}[theorem]{Lemma}
\newtheorem{proposition}[theorem]{Proposition}

\theoremstyle{definition}
\newtheorem{definition}[theorem]{Definition}

\newtheorem{remark}[theorem]{Remark}

\numberwithin{equation}{section}

\colorlet{darkblue}{blue!90!black}
\colorlet{darkgreen}{green!50!black}

\begin{document}

\title{Ergodicity for SPDEs driven by divergence-free transport noise}

\author{
	Benjamin Gess$^{1}$, 
	Rishabh S.\ Gvalani$^{2}$\orcidlink{0000-0003-2078-3383}, 
	Adrian Martini$^{3}$\orcidlink{0000-0001-9350-1338}
	}

\institute{
	Technische Universit\"{a}t Berlin \& Max-Planck-Institut f\"{u}r Mathematik in den Naturwissenschaften,
	Email: \href{mailto:benjamin.gess@tu-berlin.de}{\color{black}\texttt{benjamin.gess@tu-berlin.de}} 
	\and 
	D-MATH, ETH Z\"{u}rich, Email: \href{mailto:rgvalani@ethz.ch}{\color{black} \texttt{rgvalani@ethz.ch}}
	\and 
	Technische Universit\"{a}t Berlin, Email: \href{mailto:adrian.martini@tu-berlin.de}{\color{black} \texttt{adrian.martini@tu-berlin.de}}
	}
 
\maketitle

\begin{abstract}
	We study the ergodic behaviour of the McKean--Vlasov equations driven by common, divergence-free transport noise. In particular, we show that in dimension $d\geq 2$, if the noise is mixing and sufficiently strong it can enforce the uniqueness of invariant probability measures, even if the deterministic part of equation has multiple steady states.

	\vspace{1em}


\end{abstract}

\setcounter{tocdepth}{2}
\tableofcontents


\section{Introduction}
\label{sec:intro}
We are interested in the long-time behaviour of solutions to the following class of SPDEs:
\begin{equation}\label{eq:SPDE_intro}
	\partial_{t}\rho=\diffu\Delta\rho+\vdiv(\rho(\nabla W \ast \rho))+\sqrt{2}\str\vdiv(\rho\circ\xi)\,,
\end{equation}
posed on the $d\geq2$-dimensional unit torus $\mathbb{T}^d$, where $\diffu>0$ is the \emph{diffusivity}, $W\from\mathbb{T}^d \to \mathbb{R}$ is a sufficiently regular \emph{interaction potential}, $\str>0$ denotes the \emph{noise intensity}, and~$\xi$ is a divergence-free Gaussian noise which is white-in-time and coloured in space. The equation is to be interpreted in the Stratonovich sense. The noise~$\xi$ can be formally represented as
\begin{equation*}
	\xi(t,x) = \sum_{k \in \N }\sigma_k(x) \dot{B}_t^k \, ,		
\end{equation*}
where $( B^k )_{k \in \N}$ is a family of independent standard Brownian motions, 
and $( \sigma_k )_{k \in \N}$ is a family of smooth, divergence-free vector fields on $\mathbb{T}^d$ that form an orthonormal basis of the Cameron--Martin space associated to $\xi$. 

We study this equation for initial data $\rho_0 \in \init{0}$, where
\begin{equation*}
	\init{0}\coloneqq \left\{\rho_0 \in L^{2}(\T^d): \rho_0 \geq 0~\text{a.e. and } \int_{\T^d}\rho_0  \,\dd x=1\right\} \, .
\end{equation*}
The above SPDE can be derived as the mean-field limit of weakly interacting diffusions subject to common noise. More precisely, consider the system of SDEs for $i=1, \ldots, N$ given by 
\begin{equation*}
	\dd X_t^i = - \frac{1}{N} \sum_{\substack{j=1\\j\neq i}}^N \nabla W (X_t^i - X_t^j)  \,\dd t + \,\sqrt{2\diffu}\dd \tilde{B}_t^i +\sqrt{2} \str \sum_{k \in \N} \sigma_k(X_t^i) \circ \,\dd B_t^k \, , 
\end{equation*}
where $( \tilde{B}^i )_{i=1}^N$ is a family of independent standard Brownian motions, independent of $( B^k )_{k \in \N}$. Under suitable assumptions on $W$ and $( \sigma_k )_{k \in \N}$, it can be shown that as $N \to \infty$, the empirical measure
\begin{equation*}
	\mu_t^N := \frac{1}{N} \sum_{i=1}^N \delta_{X_t^i} \, ,
\end{equation*}
converges (in an appropriate sense) to the solution of the SPDE~\eqref{eq:SPDE_intro}, see for example~\cite{coghi_flandoli_16,coghi_gess_19,nikolaev_25}.
\paragraph{The McKean--Vlasov equation and phase transitions.}
When $K=0$, the SPDE~\eqref{eq:SPDE_intro} reduces to the following deterministic PDE:
\begin{equation}\label{eq:McKean--Vlasov}
	\partial_{t}\rho=\diffu\Delta\rho+\vdiv(\rho(\nabla W \ast \rho))\,,
\end{equation}
which is commonly referred to as the McKean--Vlasov equation. The above equation has received a lot of attention in the last two decades (see~\cite{carrillo_mccann_villani_03,chayes_panferov_10,jabin_wang_18,carrillo_gvalani_pavliotis_schlichting_20}) and its various features have been well studied: well-posedness, regularity of solutions, long-time behaviour, derivation from the underlying particle system, etc. The equation~\eqref{eq:McKean--Vlasov} is especially interesting because it is one of the canonical examples of gradient flows in the space of probability measures with respect the $2$-Wasserstein metric. The free-energy functional associated to the gradient flow structure can be expressed in the following form:
\begin{equation}\label{eq:free_energy_functional}
	E[\rho]\defeq \diffu \int_{\T^d} \rho(x) \log \rho(x) \,\dd x + \frac{1}{2} \int_{\T^d} \int_{\T^d} W(x-y) \rho(x) \rho(y) \,\dd x \,\dd y \, ,
\end{equation}
where we have slightly abused notation by identifying $\rho$ with its Lebesgue density. Since~\eqref{eq:McKean--Vlasov} is a gradient flow, its steady states are in one-to-one correspondence with the critical points of $E$ (see~\cite[Prop.~2.4]{carrillo_gvalani_pavliotis_schlichting_20}). Furthermore, one can readily check that the constant function $\unif \equiv 1 \in \init{0}$ is always a steady state of~\eqref{eq:McKean--Vlasov} and thus a critical point of $E$. However, it is well-known since the work of Chayes and Panferov~\cite[Prop.~2.9]{chayes_panferov_10} that the energy~$E$ and thus the McKean--Vlasov equation can exhibit the phenomenon of phase transitions. In particular, depending on the choice of interaction potential~$W$, there may exist some $\diffu_{\crit}>0$, such that for $\diffu>\diffu_{\crit}$, $\unif$ is the unique minimiser of~$E$, while for $\diffu<\diffu_{\crit}$, the minimisers of~$E$ are distinct from~$\unif$. 
This in turn implies that for $\diffu<\diffu_{\crit}$, the McKean--Vlasov equation~\eqref{eq:McKean--Vlasov} possesses multiple steady states. Indeed, one can derive a sharp criterion for the existence of different steady states:
\begin{proposition}[{\cite[Prop.~2.9]{chayes_panferov_10}}]\label{prop:phase_transition}
	Assume $W\in C^2(\T^d)$ is even. Then, the following two statements are equivalent:
	\begin{enumerate}
		\item There exists $k \in \ZZ^d\setminus\{0\}$ such that $\hat{W}(k)<0$.
		\item There exists $\diffu_{\crit}>0$ such that for all $\diffu<\diffu_{\crit}$, \eqref{eq:McKean--Vlasov} has a steady state which is not equal to~$\unif$. 
	\end{enumerate}
\end{proposition}

Returning to our SPDE~\eqref{eq:SPDE_intro}, one can see that the uniform state~$\unif$ is also invariant for the stochastic dynamics. 
This implies that the measure $\delta_{\unif} \in \mathcal{P}(\init{0})$ is invariant for the associated Markov process. The presence of multiple steady states for~\eqref{eq:McKean--Vlasov} then poses an interesting question for the associated stochastic dynamics:

\emph{Does the presence of the noise enforce uniqueness of invariant probability measures for~\eqref{eq:SPDE_intro} even in regimes where the deterministic dynamics possesses multiple steady states?}

In this work, we answer this question in the affirmative for a large class of mixing noises and sufficiently large noise intensities $\str$.
\paragraph{A finite-dimensional example.}
To illustrate why such a uniqueness result may be expected, we consider the following finite-dimensional example discussed by Arnold, Crauel and Wihstutz (see~\cite{arnold_crauel_wihstutz_83}). Let $X^{x}$ be the solution to the following SDE on $\R^2$:
\begin{equation}\label{eq:finite_dimensional_example}
	\dd X_t^x= A \cdot X_t^x \,\dd t + \sqrt{2}\str R \cdot X_t^x \circ \dd B_t\, , \qquad X_0^x = x \in \R^2 \, ,
\end{equation}
where $A$ is a $2\times 2$ diagonal matrix with a positive top eigenvalue,~$R$ is the skew-symmetric matrix given by
\begin{equation*}
	R
	=
	\begin{pmatrix}
			0 & 1 \\
			-1 & 0
	\end{pmatrix}\, ,
\end{equation*}
and $B$ is a one-dimensional standard Brownian motion. 
It is clear that if $K=0$ the above system is unstable. However, in the presence of noise, stability can be restored. Indeed, the authors of~\cite{arnold_crauel_wihstutz_83} showed that the top Lyapunov exponent of the above system defined by
\begin{equation*}
	\lambda_{\textnormal{top}}^{(\str)} \defeq \lim_{t \to \infty} \frac{1}{t} \log \sup_{\abs{x}\leq 1}\abs{X_{t}^{x}} 
\end{equation*}
admits a characterisation via a Furstenberg--Khasminskii formula
\begin{equation*}
	\lambda_{\textnormal{top}}^{(\str)} = \int_{\mathbb{S}^1} \ip{A s, s} \, \mu_{\str}(\dd s) \, ,
\end{equation*}
where $\mu_{\str}$ is the unique invariant probability measure of the induced dynamics on the unit~circle~$\mathbb{S}^1$ via $s_t\defeq X_{t}^{x}/\abs{X_{t}^{x}}$. They also showed that as $\str \to \infty$, $\mu_{\str}$ converges weakly to the normalised Lebesgue measure on $\mathbb{S}^1$ which in turn implies that
\begin{equation*}
	\lim_{\str \to \infty} \lambda_{\textnormal{top}}^{(\str)} = \frac{1}{2} \trace(A) \,.
\end{equation*}
It is now possible to choose $A$ such that $\trace(A)<0$ even though the top eigenvalue of $A$ is positive. In this case, for sufficiently large noise intensities~$\str$, the top Lyapunov exponent $\lambda_{\textnormal{top}}^{(\str)}$ becomes negative which implies that the system is stable.

As an example of such a matrix consider
\begin{equation*}
	A = \begin{pmatrix}
		1 & 0 \\
		0 & -2
	\end{pmatrix} \, ,
\end{equation*}
which has top eigenvalue $1$ but $\trace(A) = -1<0$. The heuristic explanation for this phenomenon is as follows: due the strong rotation effect induced by the skew-symmetric noise when $\str$ is large, the angular process $s_t$ explores the unit circle very fast and samples the directions of compression (stability) and expansion (instability) with equal probability. Thus, on average, the system experiences a net effect given by the trace of $A$ which is negative in this case. We refer to the schematic in Figure~\ref{fig:finite_dimensional_example}
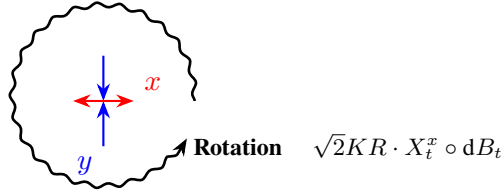
\begin{figure}
	\centering
	\begin{tikzpicture}[scale=0.4,>=Stealth,baseline=(current bounding box.center)]
		\draw[->,red,thick] (0,0) -- (1,0) node[above right] {$x$};
		\draw[->,red,thick] (0,0) -- (-1,0) node[above right] {};

		\draw[->,blue,thick] (0,-1.5)  node[below left] {$y$} -- (0,0);
		\draw[->,blue,thick] (0,1.5) -- (0,0) node[below left]{};

		\draw[black,line width=1.0pt,decorate,
		decoration={snake,amplitude=0.3mm,segment length=3mm}] (0,0) ++ (3,0) arc[start angle=0,end angle=330,radius=3] node[right] {\footnotesize\bf Rotation \quad $\sqrt{2}\str R\cdot X_t^x\circ \dd B_t$};
		\draw[-{Stealth},black,thick] (0,0) ++ ({3*cos(330)},{3*sin(330)}) 
      	-- ++({0.3*cos(60)},{0.3*sin(60)});
      \end{tikzpicture}
	  \caption{An illustration of the stabilisation effect induced by the skew-symmetric noise for the SDE~\eqref{eq:finite_dimensional_example}. The red outward-facing arrows represent the weaker, unstable directions while the blue inward-facing arrows represent the stronger, stable directions. The black circular arrow indicates the rapid rotation induced by the noise.}
	  \label{fig:finite_dimensional_example}
\end{figure}
for an illustration of this phenomenon.
\paragraph{Infinite-dimensional case.}
One can now ask  whether a similar stabilisation effect occurs in the infinite-dimensional setting for the SPDE~\eqref{eq:SPDE_intro}. We start by linearising~\eqref{eq:SPDE_intro} around the uniform state $\unif$ which leads to the following linear SPDE:
\begin{equation}
	\partial_{t}\lin=\diffu \Delta \lin + \Delta W* \lin +\sqrt{2}\str\vdiv(\lin \circ \xi) \, , \quad \lin_{0} \in L^{2}_{0}(\T^d)\, ,
	\label{eq:linearised_SPDE}
\end{equation}
where $L^{2}_{0}(\T^d)$ is the space of mean-free square-integrable functions on $\T^d$. Note that the operator $L\lin\defeq \diffu \Delta \lin + \Delta W * \lin$ can be diagonalised using Fourier modes and its spectrum is given by
\begin{equation*}
	\operatorname{spec}(L) = \{-\abs{2\uppi k}^2 (\diffu + \hat{W}(k)) : k \in \ZZ^d_0 \} \, .	
\end{equation*}
One can see that if there exists $k \in \ZZ^d\setminus\{0\}$ such that $\hat{W}(k)<0$, then for~$\diffu$ sufficiently small, $L$ possesses positive eigenvalues, which implies that the linearised dynamics in this regime is unstable in the absence of noise. This is also in complete agreement with the multiplicity result discussed in the previous section (Proposition~\ref{prop:phase_transition}). Note however that if we assume that $W$ is at least integrable, 
then $\hat{W}(k) \to 0$ as $\abs{k} \to \infty$ by the Riemann--Lebesgue lemma. This implies that the instability can only occur in finitely many Fourier modes. We refer the reader to the schematic in Figure~\ref{fig:infinite_dimensional_example} for an illustration of what the spectrum of~$L$ typically looks like in the unstable regime. Comparing this new setting with the finite-dimensional example, we now have finitely many directions of expansion and infinitely many directions of compression. We would now like to find a choice of noise that allows us to produce a similar stabilisation effect as in the finite-dimensional example.
\begin{figure}[h]
	\centering
	\begin{tikzpicture}[>=Stealth]
		\def\xscale{0.65}    
		\def\yscale{0.01}    
		\def\kmax{7}         
		\def\tick{0.10}      

		\draw[->,thick] (-0.6*\xscale,0) -- ({(\kmax+0.9)*\xscale},0);

		\foreach \k in {0,...,\kmax}{
		\draw[thick] ({\k*\xscale},-\tick) -- ({\k*\xscale},\tick);
		\node[below] at ({\k*\xscale},-1.8*\tick) {\scriptsize $\k$};
		}
		\node[below] at ({0.5*\kmax*\xscale},-4.0*\tick) {$|k|$};

		\pgfmathsetmacro{\Rx}{3.5*\xscale}
		\draw[thick,red] (\Rx,-1.3*\tick) -- (\Rx,1.3*\tick);
		\node[above,red] at (\Rx,1.6*\tick) {$R$};

		\draw[gray!45, line width=0.8pt]
		plot coordinates {
			(0.1*\xscale,{(-0.1*(2*pi)^2*0.1^2*(1 - 10*0.1^(-2)))*\yscale})
			(0.2*\xscale,{(-0.1*(2*pi)^2*0.2^2*(1 - 10*0.2^(-2)))*\yscale})
			(0.3*\xscale,{(-0.1*(2*pi)^2*0.3^2*(1 - 10*0.3^(-2)))*\yscale})
			(0.4*\xscale,{(-0.1*(2*pi)^2*0.4^2*(1 - 10*0.4^(-2)))*\yscale})
			(0.5*\xscale,{(-0.1*(2*pi)^2*0.5^2*(1 - 10*0.5^(-2)))*\yscale})
			(0.6*\xscale,{(-0.1*(2*pi)^2*0.6^2*(1 - 10*0.6^(-2)))*\yscale})
			(0.7*\xscale,{(-0.1*(2*pi)^2*0.7^2*(1 - 10*0.7^(-2)))*\yscale})
			(0.8*\xscale,{(-0.1*(2*pi)^2*0.8^2*(1 - 10*0.8^(-2)))*\yscale})
			(0.9*\xscale,{(-0.1*(2*pi)^2*0.9^2*(1 - 10*0.9^(-2)))*\yscale})
			(1*\xscale,{(-0.1*(2*pi)^2*1^2*(1 - 10*1^(-2)))*\yscale})
			(1.1*\xscale,{(-0.1*(2*pi)^2*1.1^2*(1 - 10*1.1^(-2)))*\yscale})
			(1.2*\xscale,{(-0.1*(2*pi)^2*1.2^2*(1 - 10*1.2^(-2)))*\yscale})
			(1.3*\xscale,{(-0.1*(2*pi)^2*1.3^2*(1 - 10*1.3^(-2)))*\yscale})
			(1.4*\xscale,{(-0.1*(2*pi)^2*1.4^2*(1 - 10*1.4^(-2)))*\yscale})
			(1.5*\xscale,{(-0.1*(2*pi)^2*1.5^2*(1 - 10*1.5^(-2)))*\yscale})
			(1.6*\xscale,{(-0.1*(2*pi)^2*1.6^2*(1 - 10*1.6^(-2)))*\yscale})
			(1.7*\xscale,{(-0.1*(2*pi)^2*1.7^2*(1 - 10*1.7^(-2)))*\yscale})
			(1.8*\xscale,{(-0.1*(2*pi)^2*1.8^2*(1 - 10*1.8^(-2)))*\yscale})
			(1.9*\xscale,{(-0.1*(2*pi)^2*1.9^2*(1 - 10*1.9^(-2)))*\yscale})
			(2*\xscale,{(-0.1*(2*pi)^2*2^2*(1 - 10*2^(-2)))*\yscale})
			(2.1*\xscale,{(-0.1*(2*pi)^2*2.1^2*(1 - 10*2.1^(-2)))*\yscale})
			(2.2*\xscale,{(-0.1*(2*pi)^2*2.2^2*(1 - 10*2.2^(-2)))*\yscale})
			(2.3*\xscale,{(-0.1*(2*pi)^2*2.3^2*(1 - 10*2.3^(-2)))*\yscale})
			(2.4*\xscale,{(-0.1*(2*pi)^2*2.4^2*(1 - 10*2.4^(-2)))*\yscale})
			(2.5*\xscale,{(-0.1*(2*pi)^2*2.5^2*(1 - 10*2.5^(-2)))*\yscale})
			(2.6*\xscale,{(-0.1*(2*pi)^2*2.6^2*(1 - 10*2.6^(-2)))*\yscale})
			(2.7*\xscale,{(-0.1*(2*pi)^2*2.7^2*(1 - 10*2.7^(-2)))*\yscale})
			(2.8*\xscale,{(-0.1*(2*pi)^2*2.8^2*(1 - 10*2.8^(-2)))*\yscale})
			(2.9*\xscale,{(-0.1*(2*pi)^2*2.9^2*(1 - 10*2.9^(-2)))*\yscale})
			(3*\xscale,{(-0.1*(2*pi)^2*3^2*(1 - 10*3^(-2)))*\yscale})
			(3.1*\xscale,{(-0.1*(2*pi)^2*3.1^2*(1 - 10*3.1^(-2)))*\yscale})
			(3.2*\xscale,{(-0.1*(2*pi)^2*3.2^2*(1 - 10*3.2^(-2)))*\yscale})
			(3.3*\xscale,{(-0.1*(2*pi)^2*3.3^2*(1 - 10*3.3^(-2)))*\yscale})
			(3.4*\xscale,{(-0.1*(2*pi)^2*3.4^2*(1 - 10*3.4^(-2)))*\yscale})
			(3.5*\xscale,{(-0.1*(2*pi)^2*3.5^2*(1 - 10*3.5^(-2)))*\yscale})
			(3.6*\xscale,{(-0.1*(2*pi)^2*3.6^2*(1 - 10*3.6^(-2)))*\yscale})
			(3.7*\xscale,{(-0.1*(2*pi)^2*3.7^2*(1 - 10*3.7^(-2)))*\yscale})
			(3.8*\xscale,{(-0.1*(2*pi)^2*3.8^2*(1 - 10*3.8^(-2)))*\yscale})
			(3.9*\xscale,{(-0.1*(2*pi)^2*3.9^2*(1 - 10*3.9^(-2)))*\yscale})
			(4*\xscale,{(-0.1*(2*pi)^2*4^2*(1 - 10*4^(-2)))*\yscale})
			(4.1*\xscale,{(-0.1*(2*pi)^2*4.1^2*(1 - 10*4.1^(-2)))*\yscale})
			(4.2*\xscale,{(-0.1*(2*pi)^2*4.2^2*(1 - 10*4.2^(-2)))*\yscale})
			(4.3*\xscale,{(-0.1*(2*pi)^2*4.3^2*(1 - 10*4.3^(-2)))*\yscale})
			(4.4*\xscale,{(-0.1*(2*pi)^2*4.4^2*(1 - 10*4.4^(-2)))*\yscale})
			(4.5*\xscale,{(-0.1*(2*pi)^2*4.5^2*(1 - 10*4.5^(-2)))*\yscale})
			(4.6*\xscale,{(-0.1*(2*pi)^2*4.6^2*(1 - 10*4.6^(-2)))*\yscale})
			(4.7*\xscale,{(-0.1*(2*pi)^2*4.7^2*(1 - 10*4.7^(-2)))*\yscale})
			(4.8*\xscale,{(-0.1*(2*pi)^2*4.8^2*(1 - 10*4.8^(-2)))*\yscale})
			(4.9*\xscale,{(-0.1*(2*pi)^2*4.9^2*(1 - 10*4.9^(-2)))*\yscale})
			(5*\xscale,{(-0.1*(2*pi)^2*5^2*(1 - 10*5^(-2)))*\yscale})
			(5.1*\xscale,{(-0.1*(2*pi)^2*5.1^2*(1 - 10*5.1^(-2)))*\yscale})
			(5.2*\xscale,{(-0.1*(2*pi)^2*5.2^2*(1 - 10*5.2^(-2)))*\yscale})
			(5.3*\xscale,{(-0.1*(2*pi)^2*5.3^2*(1 - 10*5.3^(-2)))*\yscale})
			(5.4*\xscale,{(-0.1*(2*pi)^2*5.4^2*(1 - 10*5.4^(-2)))*\yscale})
			(5.5*\xscale,{(-0.1*(2*pi)^2*5.5^2*(1 - 10*5.5^(-2)))*\yscale})
			(5.6*\xscale,{(-0.1*(2*pi)^2*5.6^2*(1 - 10*5.6^(-2)))*\yscale})
			(5.7*\xscale,{(-0.1*(2*pi)^2*5.7^2*(1 - 10*5.7^(-2)))*\yscale})
			(5.8*\xscale,{(-0.1*(2*pi)^2*5.8^2*(1 - 10*5.8^(-2)))*\yscale})
			(5.9*\xscale,{(-0.1*(2*pi)^2*5.9^2*(1 - 10*5.9^(-2)))*\yscale})
			(6*\xscale,{(-0.1*(2*pi)^2*6^2*(1 - 10*6^(-2)))*\yscale})
			(6.1*\xscale,{(-0.1*(2*pi)^2*6.1^2*(1 - 10*6.1^(-2)))*\yscale})
			(6.2*\xscale,{(-0.1*(2*pi)^2*6.2^2*(1 - 10*6.2^(-2)))*\yscale})
			(6.3*\xscale,{(-0.1*(2*pi)^2*6.3^2*(1 - 10*6.3^(-2)))*\yscale})
			(6.4*\xscale,{(-0.1*(2*pi)^2*6.4^2*(1 - 10*6.4^(-2)))*\yscale})
			(6.5*\xscale,{(-0.1*(2*pi)^2*6.5^2*(1 - 10*6.5^(-2)))*\yscale})
			(6.6*\xscale,{(-0.1*(2*pi)^2*6.6^2*(1 - 10*6.6^(-2)))*\yscale})
			(6.7*\xscale,{(-0.1*(2*pi)^2*6.7^2*(1 - 10*6.7^(-2)))*\yscale})
			(6.8*\xscale,{(-0.1*(2*pi)^2*6.8^2*(1 - 10*6.8^(-2)))*\yscale})
			(6.9*\xscale,{(-0.1*(2*pi)^2*6.9^2*(1 - 10*6.9^(-2)))*\yscale})
			(7*\xscale,{(-0.1*(2*pi)^2*7^2*(1 - 10*7^(-2)))*\yscale})
		};

		\fill[black] (0,0) circle (1.8pt);

		\foreach \k in {1,...,\kmax}{
		\pgfmathsetmacro{\y}{-0.1*(2*pi)^2*\k^2*(1 - 10*\k^(-2))}
		\ifdim \y pt > 0pt
		\draw[red, thick] ({\k*\xscale},{\y*\yscale}) 
		-- ++(-0.08,-0.08) -- ++(0.16,0.16)
		({\k*\xscale},{\y*\yscale}) -- ++(-0.08,0.08) -- ++(0.16,-0.16);
			
		\else
			\fill[blue]  ({\k*\xscale},{\y*\yscale}) circle (1.6pt);
		\fi
		}

		\draw[black, ->, line width=1.0pt, decorate, decoration={snake,amplitude=0.3mm,segment length=2mm}] ({0.3*\xscale},1.15) -- ({(6)*\xscale},1.15);

		\node[black, anchor=south west] at ({0.35*\xscale},1.18)
		{\footnotesize \textbf{Mixing}\quad $\sqrt{2}\str\vdiv(\lin\circ\xi)$};
	\end{tikzpicture}
	\caption{The typical spectrum of the operator $L$  in the unstable regime $\diffu\ll1$. The red crosses represent unstable modes while the blue dots represent stable modes. The red vertical line at $|k|=R$ denotes the threshold after which all modes are stable. The black arrow indicates the effect of the mixing noise $\xi$ which moves the solution to higher frequencies as time progresses.}
	\label{fig:infinite_dimensional_example}
\end{figure}
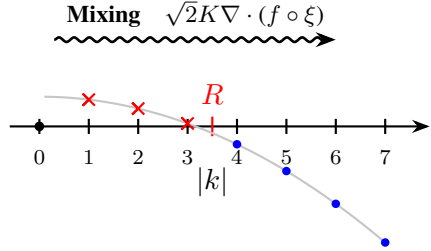
\paragraph{Passive scalar mixing.}
A natural choice of noise to make this stabilisation effect happen is to consider noises that are known to be \emph{mixing}. To understand what this means, consider the following equation describing passive-scalar transport:
\begin{equation*}
	\partial_t u =\sqrt{2 }\str \vdiv(u \circ \xi) \, , \quad u_0 \in L^{2}_{0}(\T^d) \, .
\end{equation*}
The above equation preserves all $L^{p}(\TT^{d})$-norms of the initial data $u_0$. However, for appropriate choices of $\xi$, the solution $u_t$ exhibits decay in weaker norms as $t \to \infty$ corresponding to the formation of small-scale mixing of the passive scalar. In particular, the $L^{2}(\TT^{d})$-mass of the scalar gets distributed to higher frequencies as time progresses which can be captured by the decay of negative Sobolev norms. We say that the noise $\xi$ \emph{mixes}  $u_0 \in L^{2}_{0}(\T^d)$, if we have
\begin{equation*}
	\norm{u_t}_{\dot{H}^{-s}(\T^d)} \to 0 \quad \text{ as } t \to \infty \, , \quad \text{ for some } s>0 \, ,
\end{equation*}
where $\dot{H}^{-s}(\T^d)$ is the homogeneous Sobolev space of order $-s$. Note that we have fixed $d\geq2$, since in dimension $d=1$ divergence-free vector fields are necessarily constant and can thus only translate the initial data without mixing it. 

The discussion of the above section implies in particular that the noise sends the $L^{2}(\TT^{d})$-mass of the scalar to higher frequencies as time progresses. Returning to the linearised SPDE~\eqref{eq:linearised_SPDE}, we see that if we study the angular process of this system, which lives on the unit sphere of $L^{2}(\T^d)$, then the mixing property of the noise implies that the angular process explores higher and higher frequencies as time progresses. Since the directions of expansion can only occur in finitely many low-frequency modes, we expect that if the noise is sufficiently strong it can move the solution to the stable modes at a rate faster than the rate of expansion. We now state a rough version of our main result below. The precise statement can be found in Theorem~\ref{thm:main_result}.
\begin{theorem}
	Assume $d\geq 2$, $W$ is sufficiently regular, and $\xi$ is mixing. Then, for $K$ sufficiently large, the SPDE~\eqref{eq:SPDE_intro} admits a unique invariant probability measure, namely, $\delta_{\unif}$.
\end{theorem}
\paragraph{Comments on the literature.}
The study of mixing and enhanced-dissipation phenomena in passive scalars advected by divergence-free vector fields  has been an active field of research since the pioneering work~\cite{constantin_kiselev_ryzhik_zlatos_08}.  %
Due to the extensive literature on this topic, we refer the reader to the review paper~\cite{coti_zelati_crippa_iyer_mazzucato_24} and the references therein for a comprehensive survey of the field. We only mention the papers~\cite{rowan_24,luo_tang_zhao_24,gess_yaroslavtsev_25} which are particularly relevant to this work since they study mixing and dissipation for white-in-time vector fields. We also point to related works studying the Lyapunov exponents of infinite-dimensional linear cocycles~\cite{hairer_rosati_23,hairer_punshon-smith_rosati_yi_24,chemnitz_chemnitz_25}.

SPDEs of the form~\eqref{eq:SPDE_intro} fall into a much larger class of SPDEs with conservative noise which have been studied extensively in recent years, especially in the context of fluctuating hydrodynamics and macroscopic fluctuation theory, see~\cite{giacomin_lebowitz_presutti_99,grun_mecke_rauscher_06,derrida_07,bertini_de_sole_gabrielli_jona-lasinio_landim_15,bouchet_gawedzki_nardini_16}. This includes the so-called Dean--Kawasaki equation and the stochastic thin-film equation. We refer to the papers~\cite{fehrman_gess_24,gvalani_tempelmayr_23} for the well-posedness theory of these equations,~\cite{fehrman_gess_23,djurdjevac_kremp_perkowski_24,cornalba_fischer_ingmanns_raithel_26,gess_heydecker_25} for their rigorous justification as scaling limits of interacting particle systems, and~\cite{fehrman_gess_gvalani_22,popat_wu_25} for the study of their invariant measures and ergodic properties.

McKean--Vlasov equations with common noise have been studied extensively in the literature especially in the context of mean-field games, see for instance the work~\cite{carmona_delarue_lacker_16} and the references therein. The ergodic behaviour of such equations has also been studied in recent works. In~\cite{maillet_23}, the author studies the long-time behaviour of a class of McKean--Vlasov equations driven by common noise which is constant-in-space on~$\R^d$ and shows that if the interaction and confining potentials are convex, then the common noise leads to a unique invariant measure. Note that in this setting the solutions of the deterministic McKean--Vlasov equation are known to be unique. In~\cite{delarue_tanre_maillet_24}, the authors study a more complex version of the same problem when the deterministic McKean--Vlasov equation admits multiple steady states, but for a very specific structure of the interaction. In this setting, they show that if the noise is sufficiently small the invariant measure is unique. Note that this is in contrast with our result where we need the noise to be sufficiently large to restore uniqueness of the invariant measure. Their result is based on fundamentally different effect: the constant-in-space noise and specific choice of interaction allows them to couple their system to its conditional mean whose dynamics are well-understood. 
\section{Statement of the main result}
Before stating our main result in full detail, let us first be more specific about our choice of colouring. Denote $\ZZ_{0}^{d}\defeq\ZZ^{d}\setminus\{0\}$ and let $\colour\in\ell^{2}(\ZZ^{d}_{0})$ be radially symmetric, that is $\colour_{k}=\colour_{l}$ for all $k,l\in\ZZ^{d}_{0}$ such that $\abs{k}=\abs{l}$. We then define the noise~$\xi^{\colour}$ by
\begin{equation}\label{eq:noise_def}
	\xi^{\colour}(t,x)\defeq\sum_{k\in\mathbb{Z}^{d}_{0}}\sum_{j=1}^{d-1}e_{k}(x)a_{k}^{(j)}\colour_{k}\,\dd B^{(j)}(t,k)\,,
\end{equation}
where $e_{k}(x)\defeq\euler^{2\uppi\upi\ip{k,x}}$, $(a_{k}^{(j)})_{j=1}^{d-1}$ is an orthonormal basis of $k^{\perp}\defeq\{x\in\mathbb{R}^{d}:\langle x,k\rangle=0\}$ such that $a^{(j)}_{k}=a^{(j)}_{-k}$, and $\{B^{(j)}(\place,k):k\in\mathbb{Z}_{0}^{d},\,j=1,\ldots,d-1\}$ is a collection of independent complex Brownian motions satisfying $B^{(j)}(\place,k)=\overline{B^{(j)}(\place,-k)}$ and $\EE\bigl[\abs{B^{(j)}(t,k)}^{2}\bigr]=t$ for every $t\geq0$, $k\in\ZZ^{d}_{0}$ and $j=1,\ldots,d-1$.

For notational convenience, throughout the rest of the paper we will set the diffusivity equal to~$1$. 
The general case can be recovered by rescaling time, see Remark~\ref{rem:recovering_diffusivity}. Hence, in the following we consider the SPDE
\begin{equation}\label{eq:SPDE}
	\partial_{t}\rho=\Delta\rho+\vdiv(\rho(\nabla W \ast \rho))+\sqrt{2}\str\vdiv(\rho\circ\xi^{\colour})\,.
\end{equation}

Let $\eta\in(0,1)$ and denote the set of \emph{unstable modes} by 
\begin{equation*}
	\Lambda_{W}^{(\eta)}\defeq\{k\in\mathbb{Z}_{0}^{d}:\eta+\hat{W}(k)<0\}\,,
\end{equation*}
which is finite by $W\in L^{1}(\TT^{d})$ and an application of the Riemann--Lebesgue lemma. Define
\begin{equation}\label{eq:maximal_growth_rate}
	C_{W}^{(\eta)}\defeq\max_{k\in\Lambda_{W}^{(\eta)}}\abs{\eta+\hat{W}(k)}\abs{k}^{2}<\infty
\end{equation}
and the dimension-dependent constant~$C_{d}$ by
\begin{equation}\label{eq:dimension_constant}
	C_{d}
	\defeq
	\frac{d-1}{d}
	\times
	\begin{cases}
		\frac{1}{32} & d=2\\
		\frac{3}{160} & d=3\\
		\frac{d-3}{10d(d-1)} & d\geq4
	\end{cases}
	\,.
\end{equation}
For every $\reg\in\RR$, we denote
\begin{equation*}
	h^{\reg}(\ZZ^{d}_{0})\defeq\Bigl\{\colour=(\colour_{k})_{k\in\ZZ^{d}_{0}}:\norm{\colour}_{h^{\reg}(\ZZ^{d}_{0})}^{2}\defeq\sum_{k\in\ZZ^{d}_{0}}\abs{k}^{2\reg}\abs{\colour_{k}}^{2}<\infty\Bigr\}\,.
\end{equation*}

We can now state our main result.
\begin{theorem}\label{thm:main_result}
	Let $d\geq2$, $W\in C^{2}(\TT^{d})$, 
	$\eta\in(0,1)$,
	$\reg>4$, $\colour\in h^{\reg}(\ZZ_{0}^{d})$ be non-trivial and radially symmetric,
	and $\str>0$ be sufficiently large such that 
	\begin{equation*}
		C_{W}^{(\eta)}-(1-\eta)<\norm{\colour}_{h^{-1}(\ZZ^{d}_{0})}^{2}C_{d}\str^2\,.
	\end{equation*}
	Then the only invariant probability measure of~\eqref{eq:SPDE} is given by~$\delta_{\unif}$.
\end{theorem}
The full proof of Theorem~\ref{thm:main_result} is given in Subsection~\ref{subsec:uniqueness}. Here, we give a brief overview of the main technical obstructions we need to overcome as well as the strategy we employ to do so.
\paragraph{Technical obstructions.}
The usual approach to proving uniqueness of invariant probability measures for SPDEs relies on Doob's theorem (see~\cite[Thm.~4.2.1]{da_prato_zabzcyk_96}) for which one has to establish that the associated Markov semigroup is strong Feller and irreducible.  In our setting however, this approach fails. Firstly, since the uniform state~$\unif$ is invariant under the flow of the SPDE~\eqref{eq:SPDE}, the process cannot be irreducible. Secondly, for particular choices of the noise and interaction potential, one can also show that the Markov semigroup is not strong Feller. Consider the following example: let $d=2$, $W\equiv 0$, and choose 
\begin{equation*}
	\xi(t,x)\defeq \sum_{k=1}^{4} \sigma_{k}(x)\dot{B}_{t}^{k}\, ,
\end{equation*}
with
\begin{gather*}
	\sigma_1(x) = \begin{pmatrix}   \sin(2\pi x_2)\\0 \end{pmatrix}, \quad \sigma_2(x) = \begin{pmatrix}  \cos(2\pi x_2)\\0 \end{pmatrix},\\
	\sigma_3(x) = \begin{pmatrix}0\\ \sin(2\pi x_1)  \end{pmatrix}, \quad \sigma_4(x) = \begin{pmatrix} 0 \\ \cos(2\pi x_1) \end{pmatrix} \, .
\end{gather*}
The SPDE~\eqref{eq:SPDE_intro} with this choice of noise and interaction is referred to as the 4-modes model. In a recent paper~\cite{chemnitz_chemnitz_25}, Chemnitz and Chemnitz derived the following almost-sure asymptotic lower bound for the solution to the 4-modes model
\begin{equation}\label{eq:chemnitz_lower_bound}
	\lim_{t\to\infty} \frac{1}{t} \log \norm{\rho_t}_{L^2(\T^2)}  \geq -(2\pi)^2\left(\diffu +\frac12\right)  \, ,
\end{equation}
for any initial condition $\rho_0\in L^{2}_{0}(\TT^{d})\setminus\{0\}$,  
which in particular implies the Batchelor scale conjecture for this model.   By adding $\unif$ to both the initial data and the solution, we can extend~\eqref{eq:chemnitz_lower_bound} to initial data in $\init{0}\setminus\{\unif\}$. Using~\eqref{eq:chemnitz_lower_bound}, one can argue that $F\from\init{0}\to \R$ defined as
\begin{equation*}
	F(\rho)
	\defeq
	\begin{cases}
		1 & \text{ if }  \rho= \unif\\
		0 & \text{ else }
	\end{cases}
	\, ,
\end{equation*}
is not regularised by the Markov semigroup. Indeed, if $\rho_0\neq \unif$, then we have for any $t>0$,
\begin{equation*}
	P_t F(\rho_0) = \E[F(\rho_t)|\rho_0]=0 \, ,
\end{equation*}
since if it were not, there would exist some $t>0$ such that with positive probability $\rho_t=\unif$, and thus $\rho_{t'}=\unif$ for all future times $t'\geq t$ as well, contradicting~\eqref{eq:chemnitz_lower_bound}. On the other hand, $P_t F(\unif)=1$ for all $t\geq 0$. Hence, the Markov semigroup does not regularise~$F$ and is thus not strong Feller.

We now explain how we overcome the above obstructions to prove Theorem~\ref{thm:main_result}. 
\paragraph{Strategy of proof.}
The proof of Theorem~\ref{thm:main_result} proceeds in three main steps. First, we use the dissipation enhancement of the transport noise to show that the SPDE~\eqref{eq:SPDE} has a negative top Lyapunov exponent at the uniform state~$\unif$ (see Theorem~\ref{thm:Lyapunov_exponent_negative_energy_spectrum}). In particular, this implies via the stable manifold theorem that there exists an $L^{2}(\TT^{d})$-neighbourhood of~$\unif$ which is attracted to~$\unif$ (see Theorem~\ref{thm:asymptotic_stability}). Assume now that there exists an invariant probability measure $\meas\neq\delta_{\unif}$. In particular, by Lemma~\ref{lem:support_away} there exists some set~$A$ of positive $\meas$-measure which is bounded away from~$\unif$. Second, using the mixing property of the noise again, we then show that there exists a time $t_{1}>0$ such that with uniformly positive probability each trajectory starting in $A$ enters a small neighbourhood around $\unif$ in the $H^{-1}(\TT^{d})$-topology at time~$t_{1}$. Third, we use the regularisation properties of the SPDE~\eqref{eq:SPDE} to show that there exists a time $t_{2}>0$ such that with uniformly positive probability each trajectory which is in a small $H^{-1}(\TT^{d})$-neighbourhood of~$\unif$ at time~$t_{1}$ enters the $L^{2}(\TT^{d})$-stable neighbourhood of~$\unif$ at time $t_{1}+t_{2}$ (see Lemma~\ref{lem:dynamics_regularisation}). Combining these three steps, we thus obtain that there exists a time $t=t_{1}+t_{2}>0$ such that with uniformly positive probability each trajectory starting in $A$ enters the $L^{2}(\TT^{d})$-stable neighbourhood of~$\unif$ at time~$t$ and is subsequently attracted to~$\unif$. This contradicts the invariance of~$\meas$ and thus establishes Theorem~\ref{thm:main_result}. For the full details, see Subsection~\ref{subsec:uniqueness}.
\begin{remark}\label{rem:recovering_diffusivity}
	To recover diffusivities $\diffu>0$ one can rescale time. Indeed, let $\diffu>0$, $\tilde{W}\defeq\diffu W$, $\tilde{\str}\defeq\diffu^{1/2}\str$ and $\rho$ be the solution to~\eqref{eq:SPDE}. Upon rescaling time by a factor of $\diffu$, 
	we obtain that $\tilde{\rho}(t,x)\defeq\rho(\diffu t,x)$ solves
	\begin{equation}\label{eq:SPDE_diffusivity}
		\partial_t\tilde{\rho} = \diffu\Delta\tilde{\rho}+\vdiv(\tilde{\rho}(\nabla\tilde{W}\ast\tilde{\rho}))+\sqrt{2}\tilde{\str}\vdiv(\tilde{\rho}\circ\xi^{\colour})\,.
	\end{equation}
	Let $\eta\in(0,1)$ and $\diffu'\defeq\eta\diffu$. It then follows by Theorem~\ref{thm:main_result} that~\eqref{eq:SPDE_diffusivity} has a unique invariant probability measure if
	\begin{equation*}
		C_{W}^{(\eta)}-(1-\eta)<\norm{\colour}_{h^{-1}(\ZZ^{d}_{0})}^{2}C_{d}\str^{2}\,,
	\end{equation*}
	which is equivalent to
	\begin{equation*}
		C_{\tilde{W}}^{(\diffu')}-(\diffu-\diffu')<\norm{\colour}_{h^{-1}(\ZZ^{d}_{0})}^{2}C_{d}\tilde{\str}^{2}\,.
	\end{equation*}
\end{remark}
We now comment briefly on the structure of the rest of the paper.
\paragraph{Structure of the paper.} 
In Section~\ref{sec:prelim} we collect notations and conventions used throughout the paper. In Section~\ref{sec:flow_transformation} we introduce a flow transformation that allows us to transform the SPDE~\eqref{eq:SPDE} into a PDE with random coefficients, which will be used in Section~\ref{sec:solution_theory_SPDE} to establish the well-posedness of~\eqref{eq:SPDE} as a random dynamical system (RDS). Section~\ref{sec:uniqueness_invariant_measure} is devoted to the proof of our main result (Theorem~\ref{thm:main_result}). In Subsection~\ref{subsec:Lyapunov_exponents_stable_manifold} we establish the strict negativity of the top Lyapunov exponent of~\eqref{eq:SPDE} at the uniform state, which implies asymptotic stability (Theorem~\ref{thm:asymptotic_stability}). In Subsection~\ref{subsec:reachability} we prove a reachability result (Lemma~\ref{lem:reachability}), which shows that every $H^{-1}(\TT^{d})$-neighbourhood of the uniform state can be reached with positive probability uniformly in the initial data. In Subsection~\ref{subsec:regularisation} we prove a regularisation result (Lemma~\ref{lem:dynamics_regularisation}), which allows us to pass from $H^{-1}(\TT^{d})$-neighbourhoods to $L^{2}(\TT^{d})$-neighbourhoods. In Subsection~\ref{subsec:uniqueness} we combine the results of the previous subsections to prove our main result (Theorem~\ref{thm:main_result}). In Section~\ref{sec:examples} we provide explicit examples of interaction potentials~$W$ for which our main result applies. In Appendix~\ref{app:stochastic_characteristics} we discuss the well-posedness and support properties of the stochastic characteristics associated to~\eqref{eq:SPDE}. In Appendix~\ref{app:transport_equations} we collect mixing results for transport equations that are used in the proof of Lemma~\ref{lem:reachability}. Appendix~\ref{app:joint_measurability} concerns the joint measurability of the solution map of~\eqref{eq:SPDE}, which is used in establishing the existence of an RDS. Appendices~\ref{app:continuity_inversion_map}~\&~\ref{app:supports_of_measures} collect auxiliary results used throughout the paper.
\section{Preliminaries and notation}
\label{sec:prelim}
We write~$\NN$ for the natural numbers excluding zero, $\ZZ$ for the integers and $\ZZ_{0}^{d}\defeq\ZZ^{d}\setminus\{0\}$. We define the $d$-dimensional torus by $\TT^{d}\defeq\RR^{d}/\ZZ^{d}$. Throughout, $\abs{\place}$ will indicate the norm $\abs{x}=\bigl(\sum_{i=1}^{d}\abs{x_i}^{2}\bigr)^{1/2}$ on $\RR^{d}$. From now on we will write $\ip{\place,\place}$ to denote the inner product on any Hilbert space which we either specify or leave clear from the context. We denote the transpose of a matrix $M\in\RR^{d\times d}$ by $M^{\trans}$.

If~$\meas$ is a measure on a measurable space $(X,\mathcal{A})$ and $f\from X\to Y$ is a measurable map into another measurable space, we denote the \emph{pushforward measure} of~$\meas$ under~$f$ by $f_{\#}\meas$. If~$X$ denotes a random variable, we denote by $\supp(X)$ the (topological) support of its law.

We denote the closure of a set $A$ in a topological space $X$ by $\cl A$ and the Borel sigma-algebra of~$X$ by $\mathcal{B}(X)$.

For Banach spaces $X,Y$, we denote the space of bounded, linear operators by $L(X;Y)$, and write $L(X)\defeq L(X;X)$.

We define $\ell^{2}(\ZZ^{d}_{0})$ to be the space of square-summable, real-valued sequences indexed by $\ZZ^{d}_{0}$. For every $\reg\in\RR$, we define
\begin{equation*}
	h^{\reg}(\ZZ^{d}_{0})\defeq\Bigl\{\colour=(\colour_{k})_{k\in\ZZ^{d}_{0}}:\norm{\colour}_{h^{\reg}(\ZZ^{d}_{0})}^{2}\defeq\sum_{k\in\ZZ^{d}_{0}}\abs{k}^{2\reg}\abs{\colour_{k}}^{2}<\infty\Bigr\}\,.
\end{equation*}
We call $\colour \in \ell^{2}(\ZZ^{d}_{0})$ \emph{radially symmetric} if $\colour_{k}=\colour_{l}$ for all $k,l\in\ZZ^{d}_{0}$ such that $\abs{k}=\abs{l}$.

We denote the space of homeomorphisms on~$\TT^{d}$ by $\homeo$ and for every $m\in\NN$ the space of $m$-times continuously differentiable diffeomorphisms on~$\TT^{d}$ by $\diff{m}$. We equip $\homeo$ (resp.\ $\diff{m}$) with the subspace topology induced by the compact-open topology on $C(\TT^{d};\TT^{d})$ (resp.\ $C^{m}(\TT^{d};\TT^{d})$). We call an element $\flow\in\diff{1}$ volume preserving if $\abs{\det\D\flow(x)}=1$ for all $x\in\TT^{d}$, where~$\D$ denotes the total differential.
We denote the space of volume-preserving $m$-times continuously differentiable diffeomorphisms on $\TT^{d}$ by $\diffvol{m}$. It follows by the continuity of the determinant that $\diffvol{m}$ is a closed subset of $\diff{m}$, hence Polish.
For every $\flow\from[0,\infty)\to\homeo$ we denote the \emph{spatial} inverse by~$\flow^{-1}$ and for every $\flow\from[0,\infty)\to\diff{1}$ the \emph{spatial} total differential by~$\D\flow$. 

For every $k\in\ZZ^{d}$ we define the complex exponential $e_{k}\from\TT^{d}\to\mathbb{C}$ by $e_{k}(x)\defeq\euler^{2\uppi\upi\ip{k,x}}$. We denote the complex conjugate of $z\in\mathbb{C}$ by $\overline{z}$. We define for every $u\in L^{2}(\TT^{d})$ and $k\in\ZZ^{d}$, the Fourier coefficient $\hat{u}(k)$ by
\begin{equation*}
	\hat{u}(k)\defeq\int_{\TT^{d}}u(x)e_{-k}(x)\,\dd x\,,
\end{equation*}
which is extended by duality to $\mathcal{S}'(\TT^{d})$, the space of distributions on~$\TT^{d}$, and componentwise to vector-valued functions and distributions.

For every $s\in\RR$ we denote the Sobolev space of order~$s$ by
\begin{equation*}
	\begin{split}
		H^{s}(\TT^{d})\defeq\Bigl\{u\in\mathcal{S}'(\TT^{d}):\,&\hat{u}(k)=\overline{\hat{u}(-k)}\text{ for all }k\in\ZZ^{d},\\
		&\norm{u}_{H^{s}(\TT^{d})}^{2}\defeq\sum_{k\in\ZZ^{d}}(1+\abs{2\uppi k}^{2})^{s}\abs{\hat{u}(k)}^{2}<\infty\Bigr\}\,.
	\end{split}
\end{equation*}

We denote by $L^{2}_{0}(\TT^{d})$ the space of mean-free, square-integrable functions on $\TT^{d}$ equipped with the $L^2(\TT^d)$-norm. For every $s\in\RR$ we denote 
\begin{equation*}
	\begin{split}
		\dot{H}^{s}(\TT^{d})\defeq\Bigl\{u\in\mathcal{S}'(\TT^{d}):\,&\hat{u}(0)=0,\,\hat{u}(k)=\overline{\hat{u}(-k)}\text{ for all }k\in\ZZ^{d}_{0},\\
		&\multiquad[3]\norm{u}_{\dot{H}^{s}(\TT^{d})}^{2}\defeq\sum_{k\in\ZZ^{d}_{0}}\abs{2\uppi k}^{2s}\abs{\hat{u}(k)}^{2}<\infty\Bigr\}\,.
	\end{split}
\end{equation*}
In particular, $\dot{H}^{0}(\TT^{d})=L^{2}_{0}(\TT^{d})$.

We define 
\begin{equation*}
	\init{0}\defeq\Bigl\{u\in L^{2}(\TT^{d}):u\geq0~\text{a.e. and }\int_{\TT^{d}}u(x)\,\dd x=1\Bigr\}
\end{equation*}
and
\begin{equation*}
	\init{1}\defeq H^{1}(\TT^{d})\cap\init{0}\,,
\end{equation*}
equipped with their respective subspace topologies. 

Let $\unif\from\TT^{d}\to\RR$ be the uniform state given by $\unif(x)=1$ for all $x\in\TT^{d}$. By an abuse of notation, we also denote the space-time uniform state by~$\unif$. We denote for every $s\in\RR$ and $r>0$,
\begin{equation*}
	B_{r}^{s}(\unif)\defeq\{u\in H^{s}(\mathbb{T}^{d}):\hat{u}(0)=1,\,\lVert u-\unif\rVert_{\dot{H}^{s}(\mathbb{T}^{d})}<r\}\,.
\end{equation*}

We denote for every $U\subset L^{2}(\TT^{d})$, $\diam(U)\defeq\sup_{u,v\in U}\norm{u-v}_{L^{2}(\TT^{d})}$.

We define the Leray projection $\Pi\from L^{2}(\TT^{d};\RR^{d})\to L^{2}(\TT^{d};\RR^{d})$ by
\begin{equation*}
	(\Pi f)(x)\defeq\sum_{k\in\ZZ_{0}^{d}}e_{k}(x)P_{k}\hat{f}(k)\,,\qquad P_{k}\defeq1-\frac{k\otimes k}{\abs k^{2}}\,,
\end{equation*}
and for every $\colour\in\ell^{2}(\ZZ^{d}_{0})$ the Fourier multiplier $\colour(D)\from L^{2}(\TT^{d};\RR^{d})\to L^{2}(\TT^{d};\RR^{d})$ by
\begin{equation*}
	(\colour(D)f)(x)\defeq\sum_{k\in\ZZ_{0}^{d}}e_{k}(x)\colour_{k}\hat{f}(k)\,.
\end{equation*}
For every $\colour\in\ell^2(\ZZ^{d}_{0})$ we define
\begin{equation}\label{eq:Cameron--Martin_space}
	\cm(\colour)\defeq L^{2}\bigl([0,T];\colour(D)\Pi L^{2}(\TT^{d};\RR^{d})\bigr)\,,
\end{equation}
which is the Cameron--Martin space of the noise~$\xi^{\colour}$ defined in~\eqref{eq:noise_def}.

Let~$H$ be some separable Hilbert space such that the embedding $\ell^{2}(\ZZ_{0}^{d};\RR^{d-1})\embed H$
is Hilbert--Schmidt. Let 
\begin{equation*}
	\Omega\defeq C_{0}(\RR;H)\defeq\{\omega\in C(\RR;H):\omega(0)=0\}\,,
\end{equation*}
be equipped with the compact-open topology,
\begin{equation*}
	\mathcal{F}^{0}\defeq\mathcal{B}(C_{0}(\RR;H))=\sigma(\omega(t):t\in\RR)\,,
\end{equation*}
and~$\PP$ be the law of the cylindrical Wiener process on $(\Omega,\mathcal{F}^{0})$,
which can be obtained by gluing two independent cylindrical Wiener
processes to each other at $t=0$. 

For every $t\in\RR$ we define the \emph{shift}
\begin{equation*}
	\shift_{t}\from\Omega\to\Omega\,,\qquad\shift_{t}\omega(s)\defeq\omega(s+t)-\omega(t)\,,\quad s\in\RR\,.
\end{equation*}
It then follows by~\cite[Prop.~4.1]{fehrman_gess_gvalani_22},
that $(\Omega,\mathcal{F}^{0},\PP,(\shift_{t})_{t\in\RR})$ is a \emph{metric
dynamical system} (MDS) in the sense of~\cite[Def.~3]{arnold_scheutzow_95}. In particular, $\PP$ is $(\shift_{t})_{t\in\RR}$-invariant. Let~$\mathcal{F}$ be the $\PP$-completion of~$\mathcal{F}^{0}$
and denote for every $s\leq t$ by~$\mathcal{F}_{s,t}$ the $\PP$-completion of
\begin{equation*}
	\sigma(\omega(u)-\omega(v):s\leq u\leq v\leq t)\,.
\end{equation*}
One can then show (see e.g.~\cite[Sec.~4]{fehrman_gess_gvalani_22}),
that
\begin{equation*}
	\shift_{u}^{-1}\mathcal{F}_{s,t}=\mathcal{F}_{s+u,t+u}\,,
\end{equation*}
which implies that $(\Omega,\mathcal{F}^{0},\PP,(\shift_{t})_{t\in\RR},(\mathcal{F}_{s,t})_{s\leq t})$
is a \emph{filtered dynamical system} (FDS) in the sense of~\cite[Def.~8]{arnold_scheutzow_95}. We further set $\mathcal{F}_{t}\defeq\mathcal{F}_{0,t}$.

Let~$X$ be a topological
space. A \emph{continuous (one-sided) random dynamical system (RDS)
(over $(\Omega,\mathcal{F}^{0},\PP,(\shift_{t})_{t\in\RR})$)} is
a map $\flow\from[0,\infty)\times\Omega\times X\to X$ such that
\begin{enumerate}
	\item $\flow$ is $\bigl(\mathcal{B}([0,\infty))\otimes\mathcal{F}^{0}\otimes\mathcal{B}(X),\mathcal{B}(X)\bigr)$-measurable;
	\item For all $s,t\in[0,\infty)$, $x\in X$ and $\omega\in\Omega$ it holds
	that $\flow(0,\omega,x)=x$ and
	\begin{equation*}
		\flow(s+t,\omega,x)=\flow(t,\shift_{s}\omega,\flow(s,\omega,x))\,;
	\end{equation*}
	\item For each $\omega\in\Omega$,
	\begin{equation*}
		\flow(\place,\omega,\place)\from[0,\infty)\times X\to X\,,\qquad(t,x)\mapsto\flow(t,\omega,x)
	\end{equation*}
	is continuous.
\end{enumerate}
Assume, in addition, that~$X$ is a $d$-dimensional, smooth manifold
and $m\in\NN$. A $C^{m}(X;X)$-RDS is a continuous RDS $\flow$ such
that for each $(t,\omega)\in[0,\infty)\times\Omega$, the map
\begin{equation*}
	\flow(t,\omega,\place)\from X\to X\,,\qquad x\mapsto\flow(t,\omega,x)
\end{equation*}
is $m$-times continuously differentiable in $x$ with derivatives
continuous in $(t,x)$. See~\cite{arnold_98} for a comprehensive treatment.

Let $m\in\NN$ and $(\flow_{t})_{t\geq0}$ be a stochastic process over $(\Omega,\mathcal{F}_{0},\PP)$. We call $(\flow_{t})_{t\geq0}$ a \emph{stochastic flow of homeomorphisms (resp.\ $C^{m}(\TT^{d};\TT^{d})$-diffeomorphisms and volume-preserving $C^{m}(\TT^{d};\TT^{d})$-diffeomorphisms)} if for every $\omega\in\Omega$, it holds that $\flow(\omega)\in C([0,\infty);\homeo)$ (resp.\ $\flow(\omega)\in C([0,\infty);\diff{m})$ and $\flow(\omega)\in C([0,\infty);\diffvol{m})$) such that $\flow_{0}(\omega)=\mathrm{id}$. See~\cite{kunita_90} for a comprehensive treatment.
\section{The flow transformation}\label{sec:flow_transformation}
For our eventual application of the stable manifold theorem, we need to show that the SPDE~\eqref{eq:SPDE} induces a random dynamical system. To this end, we first introduce the so-called \emph{flow transformation} of~\eqref{eq:SPDE}, which transforms the SPDE into a random PDE driven by a stochastic flow of volume-preserving diffeomorphisms induced by the stochastic characteristics. We then show that the flow transformation admits a pathwise unique weak solution which depends continuously on the initial data and the driving flow of diffeomorphisms (Lemma~\ref{lem:well_posedness_flow_transformation_deterministic}). This allows us to define a random dynamical system associated to~\eqref{eq:SPDE} in Subsection~\ref{subsec:RDS_for_SPDE}.

Let $\reg>2$, $\colour\in h^{\reg}(\ZZ_{0}^{d})$ be radially symmetric and $\str>0$. Denote by $\flow\from[0,\infty)\times\Omega\times\TT^{d}\to\TT^{d}$ the stochastic characteristics induced by the SDE,
\begin{equation}\label{eq:stochastic_characteristics_SDE}
	\dd X_{t}=-\sqrt{2}\str\sum_{k\in\mathbb{Z}_{0}^{d}}\sum_{j=1}^{d-1}\colour_{k}a_{k}^{(j)}e_{k}(X_{t})\circ\dd B^{(j)}(t,k)\,,\qquad X_{0}=x\in\TT^{d}\,,
\end{equation}
which exist by an application of Lemma~\ref{lem:RDS_for_stochastic_characteristics} as a stochastic flow of volume-preserving $C^{2}(\TT^{d};\TT^{d})$-diffeomorphisms and a random dynamical system. In particular, $\flow(\omega)\defeq\flow(\place,\omega,\place)\in C([0,\infty);\diffvol{2})$ for every $\omega\in\Omega$.

Recall that we denote the spatial inverse by
\begin{equation*}
	\flow^{-1}=(\flow^{-1,1},\ldots,\flow^{-1,d})\from\Omega\times[0,\infty)\times\TT^{d}\to\TT^{d}\,.
\end{equation*}

Define
\begin{equation*}
	A(t,\flow(\omega),x)\defeq(\D\flow(t,\omega,x))^{-1}((\D\flow(t,\omega,x))^{-1})^{\trans}\,,
\end{equation*}
\begin{equation*}
	b^{i}(t,\flow(\omega),x)\defeq(\Delta\flow^{-1,i})(t,\omega,\flow(t,\omega,x))\,,\qquad i=1,\ldots,d\,,
\end{equation*}
and 
for every measurable function $v\from\Omega\times[0,\infty)\times\TT^{d}\to\RR$,
\begin{equation*}
	((\nabla W)\ast_{\flow}v)(t,\omega,x)\defeq\int_{\TT^{d}}(\nabla W)(\flow(t,\omega,x)-\flow(t,\omega,y))v(t,\omega,y)\,\dd y\,.
\end{equation*}
Let $(\omega,t,x)\mapsto\rho(t,x;\omega,\rho_{0})$ be a solution to~\eqref{eq:SPDE} with initial data $\rho_{0}\in L^{2}(\TT^{d})$ and define for every fixed $\omega\in\Omega$, $t\geq0$, $x\in\TT^{d}$,
\begin{equation*}
	\overline{\rho}(t,x;\flow(\omega),\rho_{0})\defeq\rho(t,\flow(t,\omega,x);\omega,\rho_{0})\,.
\end{equation*}
It then follows by an application of the It\^{o}--Wentzell formula (see e.g.~\cite[Prop.~2.14]{agresti_sauerbrey_veraar_25}), that $\overline{\rho}(t,x)\defeq\overline{\rho}(t,x;\flow(\omega),\rho_{0})$ solves the path-by-path PDE,
\begin{equation}\label{eq:path_by_path_PDE}
	\begin{split}
		\partial_{t}\overline{\rho}(t,x)&=\trace(A(t,\flow(\omega),x)\D^{2}\overline{\rho}(t,x))\\
		&\quad+b(t,\flow(\omega),x)\cdot\D\overline{\rho}(t,x)\\
		&\quad+(\D\flow(t,\omega,x))^{-1}\colon(\D(\overline{\rho}((\nabla W)\ast_{\flow}\overline{\rho}))(t,\omega,x))^{\trans}\,.
	\end{split}
\end{equation}
We refer to~$\overline{\rho}$ as the \emph{flow transformation} of~$\rho$. Note that we can recover~$\rho$ from~$\overline{\rho}$ via the relation
\begin{equation*}
	\rho(t,x;\omega,\rho_{0})=\overline{\rho}(t,\flow^{-1}(t,\omega,x);\flow(\omega),\rho_{0})\,.
\end{equation*}
Let $T>0$. We first study the well-posedness of the PDE~\eqref{eq:path_by_path_PDE} driven by a \emph{deterministic} element~$\flow$ of $C([0,T];\diffvol{2})$. In particular, Lemma~\ref{lem:well_posedness_flow_transformation_deterministic} shows that there exists a unique weak solution to~\eqref{eq:path_by_path_PDE}, which is locally Lipschitz continuous in the initial data and locally Hölder continuous in~$\flow$. The pathwise well-posedness of~\eqref{eq:path_by_path_PDE} driven by the stochastic characteristics then follows by an application of Lemma~\ref{lem:RDS_for_stochastic_characteristics}, see Lemma~\ref{lem:well_posedness_flow_transformation}.

We first define what we understand to be a weak solution to~\eqref{eq:path_by_path_PDE} driven by a fixed~$\flow$.
\begin{definition}[{Weak Solution to~\eqref{eq:path_by_path_PDE}}]\label{def:weak_solution_flow_transformation} 
	Let $T>0$, $\rho_{0}\in\init{0}$, $W\in C^{2}(\TT^{d})$ and $\flow\in C([0,T];\diffvol{2})$.
	We call
	\begin{equation*}
		\overline{\rho}\in C([0,T];L^{2}(\TT^{d}))\cap L^{2}([0,T];H^{1}(\TT^{d}))
	\end{equation*}
	a weak solution to~\eqref{eq:path_by_path_PDE} driven by~$\flow$ 
	and with initial data~$\rho_{0}$, if for all $t\in[0,T]$ and $\psi\in H^{1}(\TT^{d})$,
	\begin{equation*}
		\begin{split}
			&\ip{\overline{\rho}_{t},\psi}_{L^{2}(\TT^{d})}\\
			&\multiquad[3]=\ip{\rho_{0},\psi}_{L^{2}(\TT^{d})}\\
			&\multiquad[3]\quad+\int_{0}^{t}\ip{\trace(A(s,\flow,\place)\D^{2}\overline{\rho}_{s}),\psi}_{H^{-1}(\TT^{d}),\,H^{1}(\TT^{d})}\,\dd s\\
			&\multiquad[3]\quad+\int_{0}^{t}\ip{b(s,\flow,\place)\cdot\D\overline{\rho}_{s},\psi}_{H^{-1}(\TT^{d}),\,H^{1}(\TT^{d})}\,\dd s\\
			&\multiquad[3]\quad+\int_{0}^{t}\ip{\bigl((\D\flow)^{-1}\colon(\D(\overline{\rho}((\nabla W)\ast_{\flow}\overline{\rho})))^{\trans}\bigr)(s,\place),\psi}_{H^{-1}(\TT^{d}),\,H^{1}(\TT^{d})}\,\dd s\,.
		\end{split}
	\end{equation*}
\end{definition}
\begin{lemma}\label{lem:well_posedness_flow_transformation_deterministic}
	Let $T>0$, $\rho_{0}\in\init{0}$, $W\in C^{2}(\TT^{d})$ and $\flow\in C([0,T];\diffvol{2})$.
	Then there exists a unique weak solution $(t,x)\mapsto\overline{\rho}(t,x;\flow,\rho_{0})$ to~\eqref{eq:path_by_path_PDE} driven by~$\flow$ and with initial data~$\rho_{0}$. For fixed~$\flow$, the map $\rho_{0}\mapsto\overline{\rho}(\place,\place;\flow,\rho_{0})$ is locally Lipschitz continuous from $\init{0}$ to $C([0,T];\init{0})$. Further, for fixed $\rho_{0}$, the map $\flow\mapsto\overline{\rho}(\place,\place;\flow,\rho_{0})$ is locally $1/2$-Hölder continuous from $C([0,T];\diffvol{2})$ to $C([0,T];\init{0})$.
\end{lemma}
\begin{proof}
	The equation is locally monotone and generalized coercive. Hence, local well-posedness and continuity in the initial data follow by a Galerkin approximation, see e.g.~\cite[Thms.~1.1~\&~.1.2]{liu_roeckner_13}. The invariance of $\init{0}$, global well-posedness and the local Hölder continuity in~$\flow$ follow by a priori energy estimates.
\end{proof}
Combining Lemma~\ref{lem:well_posedness_flow_transformation_deterministic} with Lemma~\ref{lem:RDS_for_stochastic_characteristics} we obtain the well-posedness of the flow transformation driven by the stochastic characteristics.
\begin{lemma}\label{lem:well_posedness_flow_transformation}
	Let $T>0$, $\rho_{0}\in\init{0}$, $W\in C^{2}(\TT^{d})$, $\reg>2$, $\colour\in h^{\reg}(\ZZ^{d}_{0})$ be radially symmetric, $\str>0$ and $\flow$ be the stochastic characteristics induced by~\eqref{eq:stochastic_characteristics_SDE} as constructed in Lemma~\ref{lem:RDS_for_stochastic_characteristics}. Then for every $\omega\in\Omega$ there exists a unique weak solution $(t,x)\mapsto\overline{\rho}(t,x;\flow(\omega),\rho_{0})$ to~\eqref{eq:path_by_path_PDE} driven by~$\flow(\omega)$ and with initial data~$\rho_{0}$. It follows that the map $\rho_{0}\mapsto\overline{\rho}(\place,\place;\flow(\omega),\rho_{0})$ is locally Lipschitz continuous from $\init{0}$ to $C([0,T];\init{0})$ for every $\omega\in\Omega$.
\end{lemma}
\begin{proof}
	The claim follows by combining Lemmas~\ref{lem:well_posedness_flow_transformation_deterministic}~\&~\ref{lem:RDS_for_stochastic_characteristics}.
\end{proof}
\section{Solution theory for the SPDE}\label{sec:solution_theory_SPDE}	
In this section we study the well-posedness of the SPDE~\eqref{eq:SPDE}. To this end, we first show in Subsection~\ref{subsec:well_posedness_SPDE} that there exists a pathwise-unique weak solution to~\eqref{eq:SPDE}, which can be represented via the flow transformation~\eqref{eq:path_by_path_PDE} driven by the stochastic characteristics induced by~\eqref{eq:stochastic_characteristics_SDE}. In Subsection~\ref{subsec:regularity_SPDE} we then study regularity properties of the SPDE. In Subsection~\ref{subsec:continuity_SPDE_wrt_stochastic_characteristics} we establish the continuity of the solution map with respect to the stochastic characteristics. We then use this continuity to show in Subsection~\ref{subsec:support_theorem_SPDE} a support theorem and in Subsection~\ref{subsec:RDS_for_SPDE} the existence of a random dynamical system associated to~\eqref{eq:SPDE}.
\subsection{Well-posedness}\label{subsec:well_posedness_SPDE}
In this subsection we show that there exists a pathwise-unique weak solution to the SPDE~\eqref{eq:SPDE}.
\begin{definition}[{Weak Solution to~\eqref{eq:SPDE}}]\label{def:weak_solution_SPDE}
	Let $(\Omega,\mathcal{F},(\mathcal{F}_{t})_{t\geq0},\PP)$ be as in Section~\ref{sec:prelim}, $T>0$, $W\in C^{2}(\TT^{d})$, $\colour\in\ell^{2}(\ZZ^{d}_{0})$ be radially symmetric, $\str>0$ and $\rho_{0}\in L^{2}(\TT^{d})$. 
	We call a progressively measurable map $\rho\from\Omega\times[0,T]\times\TT^{d}\to\RR$ a weak solution
	to~\eqref{eq:SPDE} with initial data $\rho_{0}$, if it satisfies $\PP\text{-a.s.}$
	\begin{equation*}
		\rho\in C([0,T];L^{2}(\TT^{d}))\cap L^{2}([0,T];H^{1}(\TT^{d}))\,,
	\end{equation*}
	and for all $\psi\in H^{1}(\TT^{d})$, $\PP$-a.s. the following identity
	holds for all $t\in[0,T]$,
	\begin{equation*}
		\begin{split}
			&\ip{\rho_{t},\psi}_{L^{2}(\TT^{d})}=\ip{\rho_{0},\psi}_{L^{2}(\TT^{d})}-\int_{0}^{t}\int_{\TT^{d}}\nabla\rho_{s}(x)\cdot\nabla\psi(x)\,\dd x\,\dd s\\
			&\multiquad[4]-\int_{0}^{t}\int_{\TT^{d}}\rho_{s}(x)(\nabla W\ast\rho_{s})(x)\cdot\nabla\psi(x)\,\dd x\,\dd s\\
			&\multiquad[4]-\sqrt{2}\str\sum_{k\in\mathbb{Z}_{0}^{d}}\sum_{j=1}^{d-1}\int_{0}^{t}\int_{\TT^{d}}\rho_{s}(x)\colour_{k}e_{k}(x)a_{k}^{(j)}\cdot\nabla\psi(x)\,\dd x\circ\dd B^{(j)}(s,k)\,.
		\end{split}
	\end{equation*}
\end{definition}
\begin{lemma}\label{lem:well_posedness_SPDE}
	Let $T>0$, $\rho_{0}\in\init{0}$, $W\in C^{2}(\TT^{d})$, $\reg>2$, $\colour\in h^{\reg}(\ZZ^{d}_{0})$ be radially symmetric, $\str>0$, $\flow$ be the stochastic characteristics induced by~\eqref{eq:stochastic_characteristics_SDE} as constructed in Lemma~\ref{lem:RDS_for_stochastic_characteristics} and $\overline{\rho}$ be the weak solution to the path-by-path PDE~\eqref{eq:path_by_path_PDE} driven by~$\flow$ and with initial data~$\rho_{0}$ as constructed in Lemma~\ref{lem:well_posedness_flow_transformation}.
	Then there exists a pathwise-unique weak solution to~\eqref{eq:SPDE}
	in the sense of Definition~\ref{def:weak_solution_SPDE}, which is
	given by 
	\begin{equation}\label{eq:SPDE_solution_via_flow_transformation}
		(\omega,t,x)\mapsto\rho(t,x;\omega,\rho_{0})\defeq\overline{\rho}(t,\flow^{-1}(t,\omega,x);\flow(\omega),\rho_{0})\,.
	\end{equation}
	The map $\rho_{0}\mapsto\rho(\place,\place;\omega,\rho_{0})$ is locally Lipschitz continuous from $\init{0}$ to $C([0,T];\init{0})$ for every $\omega\in\Omega$.
\end{lemma}
\begin{proof}
	Let $(t,\omega,x)\mapsto\flow(t,\omega,x)$ be the stochastic characteristics as constructed in Lemma~\ref{lem:RDS_for_stochastic_characteristics}, $\overline{\rho}(t,x;\flow(\omega),\rho_0)$ be the weak solution to the path-by-path PDE~\eqref{eq:path_by_path_PDE} driven by $\flow$ and with initial data $\rho_{0}$ as constructed in Lemma~\ref{lem:well_posedness_flow_transformation} and let $(\omega,t,x)\mapsto\rho(t,x;\omega,\rho_{0})$ be given by~\eqref{eq:SPDE_solution_via_flow_transformation}.

	It then follows by an explicit calculation, that~$\rho$ is a weak solution to~\eqref{eq:SPDE}. To show pathwise uniqueness, it suffices to use the property that transport noise leaves the $L^{2}(\TT^d)$-norm invariant and apply pathwise energy estimates to establish pathwise $L^{2}(\TT^{d})$-stability.
	
	The local Lipschitz continuity of the solution in the initial data for every $\omega\in\Omega$ is a direct consequence of~\eqref{eq:SPDE_solution_via_flow_transformation}, Lemma~\ref{lem:well_posedness_flow_transformation} and Lemma~\ref{lem:RDS_for_stochastic_characteristics}.
\end{proof}
\subsection{Regularity}\label{subsec:regularity_SPDE}
For every $t\geq0$ define the pushforward measure $\rho(t,\place;\place,\rho_{0})_{\#}\mathbb{P}$ acting on $A\in\mathcal{B}(H^{-1}(\mathbb{T}^{d}))$ by
\begin{equation*}
	\rho(t,\place;\place,\rho_{0})_{\#}\mathbb{P}(A)\defeq\int_{\Omega}\mathds{1}_{A}(\rho(t,\place;\omega,\rho_{0}))\PP(\dd\omega)\,.
\end{equation*}
\begin{lemma}[Weak continuity]\label{lem:weak_continuity}
	Let $W\in C^{2}(\TT^{d})$, $\reg>2$, $\colour\in h^{\reg}(\ZZ^{d}_{0})$ be radially symmetric and $\str>0$.
	For every $t\geq0$, it then follows that 
	\begin{equation*}
		\init{0}\ni\rho_{0}\mapsto\rho(t,\place;\place,\rho_{0})_{\#}\mathbb{P}
	\end{equation*}
	is continuous in the sense of weak convergence of measures on $H^{-1}(\mathbb{T}^{d})$.
\end{lemma}
\begin{proof}
	Let $f\from H^{-1}(\mathbb{T}^{d})\to\RR$ be bounded and Lipschitz continuous. The (pathwise) local Lipschitz continuity of the solution in its initial
	data (cf.\ Lemma~\ref{lem:well_posedness_SPDE}) combined with the dominated convergence theorem implies that $\mathbb{E}[f(\rho(t,\place;\place,\rho_{0}))]$ is continuous
	in $\rho_{0}\in\init{0}$. The claim then follows by the Portmanteau theorem.
\end{proof}
In the next lemma we show that control over the $H^{-1}(\TT^{d})$-norm of the initial data yields control over the the solution.
\begin{lemma}\label{lem:H-1_bound}
	Let $\rho_{0}\in\init{0}$, $W\in C^{2}(\mathbb{T}^{d})$, $\reg>2$, $\colour\in h^{\reg}(\ZZ_{0}^{d})$ be radially symmetric and $\str>0$. Then there exists a constant $C=C(d,\norm{W}_{C^{2}(\TT^{d})},\norm{\colour}_{h^{\reg}(\ZZ_{0}^{d})},\str)$ such that for all $t>0$,
	\begin{equation*}
		\EE\bigl[\norm{\rho(t,\place;\place,\rho_{0})-\unif}_{\dot{H}^{-1}(\TT^{d})}^{2}\bigr]\leq\EE\bigl[\norm{\rho_{0}-\unif}_{\dot{H}^{-1}(\TT^{d})}^{2}\bigr]\euler^{Ct}
	\end{equation*}
	and for all $T>0$, 
	\begin{equation*}
		\int_{0}^{T}\EE\bigl[\norm{\rho(t,\place;\place,\rho_{0})-\unif}_{L^{2}(\TT^{d})}^{2}\bigr]\,\dd t\leq\EE\bigl[\norm{\rho_{0}-\unif}_{\dot{H}^{-1}(\TT^{d})}^{2}\bigr]\euler^{CT}\,.
	\end{equation*}
\end{lemma}
\begin{proof}
	The claim follows by an application of Itô's formula to the map $t\mapsto\abs{(-\Delta)^{-1/2}(\rho(t,x;\omega,\rho_{0})-1)}^2$, an energy estimate and Gronwall's inequality.
\end{proof}
\subsection{Continuity with respect to the stochastic characteristics}\label{subsec:continuity_SPDE_wrt_stochastic_characteristics}
It is clear from~\eqref{eq:SPDE_solution_via_flow_transformation} and Lemma~\ref{lem:well_posedness_SPDE} that the solution $\rho$ to~\eqref{eq:SPDE} only depends on~$\omega$ through the stochastic characteristics $\flow(\omega)=\flow(\place,\omega,\place)\in C([0,T];\diffvol{2})$. Hence, we define by a slight abuse of notation for every $\flow\in C([0,T];\diffvol{2})$,
\begin{equation*}
	\rho(t,x;\flow,\rho_{0})=\overline{\rho}(t,\flow^{-1}(t,x);\flow,\rho_{0})\,.
\end{equation*}
Next we show that $\rho(\place,\place;\flow,\rho_{0})$ is continuous in $\flow$.
\begin{lemma}\label{lem:continuity_flow_to_SPDE}
	Let $T>0$, $\rho_{0}\in\init{0}$ and $W\in C^{2}(\TT^{d})$.
	Then, it follows that the map $\flow\mapsto\rho(\place,\place;\flow,\rho_{0})$ is continuous from $C([0,T];\diffvol{2})$ to $C([0,T];\init{0})$.
\end{lemma}
\begin{proof}
	Since $C([0,T];\diffvol{2})$ is a metric space, it suffices to show sequential
	continuity of the map $\flow\mapsto\rho(\place,\place;\flow,\rho_{0})$. 
	Let $\flow\in C([0,T];\diffvol{2})$
	and $(\flow_{n})_{n\in\mathbb{N}}$ be a sequence such that $\flow_{n}\in C([0,T];\diffvol{2})$
	for every $n\in\mathbb{N}$ and $\flow_{n}\to\flow$ in $C([0,T];\diffvol{2})$
	as $n\to\infty$. We need to show that $\lim_{n\to\infty}\rho(\place,\place;\flow_{n},\rho_{0})=\rho(\place,\place;\flow,\rho_{0})$
	in $C([0,T];\init{0})$.

	We decompose for every $t\in[0,T]$ and $x\in\TT^{d}$,
	\begin{equation}\label{eq:undoing_flow_transformation_decomposition}
		\begin{split}
			&\rho(t,x;\flow,\rho_{0})-\rho(t,x;\flow_{n},\rho_{0})\\
			&=\bigl(\overline{\rho}(t,\flow^{-1}(t,x);\flow,\rho_{0})-\overline{\rho}(t,\flow_{n}^{-1}(t,x);\flow,\rho_{0})\bigr)\\
			&\quad+\bigl(\overline{\rho}(t,\flow_{n}^{-1}(t,x);\flow,\rho_{0})-\overline{\rho}(t,\flow_{n}^{-1}(t,x);\flow_{n},\rho_{0})\bigr)\,.
		\end{split}
	\end{equation}
	Let us consider the first term on the right-hand side of~\eqref{eq:undoing_flow_transformation_decomposition}.
	Since the flow of diffeomorphisms $\flow\in C([0,T];\diffvol{2})$ is fixed and $\flow_{n}\to\flow$
	in $C([0,T];\diff{1})$ as $n\to\infty$, we
	can apply Lemma~\ref{lem:continuity_CTL2_concatenation_diffeomorphism_flow_inverse}
	to deduce
	\begin{equation*}
		\lim_{n\to\infty}\sup_{t\in[0,T]}\norm{\overline{\rho}(t,\flow^{-1}(t,\place);\flow,\rho_{0})-\overline{\rho}(t,\flow_{n}^{-1}(t,\place);\flow,\rho_{0})}_{L^{2}(\TT^{d})}=0\,.
	\end{equation*}

	Next, let us consider the second term in~\eqref{eq:undoing_flow_transformation_decomposition}.
	Using the integral transformation theorem and the Cauchy--Schwarz
	inequality, it follows that
	\begin{equation*}
		\begin{split}
			&\sup_{t\in[0,T]}\lVert\overline{\rho}(t,\flow_{n}^{-1}(t,\place);\flow,\rho_{0})-\overline{\rho}(t,\flow_{n}^{-1}(t,\place);\flow_{n},\rho_{0})\rVert_{L^{2}(\TT^{d})}^{2}\\
			&\leq\lVert\overline{\rho}(\place,\place;\flow,\rho_{0})-\overline{\rho}(\place,\place;\flow_{n},\rho_{0})\rVert_{C([0,T];L^{2}(\TT^{d}))}^{2}\,.
		\end{split}
	\end{equation*}
	Consequently, we can use Lemma~\ref{lem:well_posedness_flow_transformation_deterministic}
	and the convergence $\lim_{n\to\infty}\flow_{n}=\flow$ in $C([0,T];\diffvol{2})$
	to deduce
	\begin{equation*}
		\begin{split}
			&\lim_{n\to\infty}\sup_{t\in[0,T]}\lVert\overline{\rho}(t,\flow_{n}^{-1}(t,\place);\flow,\rho_{0})-\overline{\rho}(t,\flow_{n}^{-1}(t,\place);\flow_{n},\rho_{0})\rVert_{L^{2}(\TT^{d})}^{2}\\
			&=\lim_{n\to\infty}\sup_{t\in[0,T]}\lVert\overline{\rho}(t,\place;\flow,\rho_{0})-\overline{\rho}(t,\place;\flow_{n},\rho_{0})\rVert_{L^{2}(\TT^{d})}^{2}\\
			&=0\,.
		\end{split}
	\end{equation*}
	This yields the claim.
\end{proof}
\subsection{The support theorem for the SPDE}\label{subsec:support_theorem_SPDE}
In this subsection we combine the support theorem for the stochastic characteristics (Theorem~\ref{thm:support_theorem_flow}) with the continuity of the solution map (Lemma~\ref{lem:continuity_flow_to_SPDE}) to show a support theorem for the SPDE~\eqref{eq:SPDE}.
\begin{theorem}\label{thm:support_theorem_McKean_Vlasov}
	Let $T>0$, $\rho_{0}\in \init{0}$, $W\in C^{2}(\TT^{d})$, $\reg>4$, $\colour\in h^{\reg}(\ZZ^{d}_{0})$ be radially symmetric, $\str>0$, and $\flow$ be the stochastic characteristics constructed in Lemma~\ref{lem:RDS_for_stochastic_characteristics}.
	Then, it holds that,
	\begin{equation*}
		\supp(\rho(\place,\place;\flow,\rho_{0}))=\cl\{\rho(\place,\place;\flow_{h},\rho_{0}):h\in\cm(\colour)\}\,,
	\end{equation*}
	where $\cl$ is taken in $C([0,T];\init{0})$ and
	each~$\flow_{h}$ solves~\eqref{eq:flow_skeleton}. In particular, each $\rho(\place,\place;\flow_{h},\rho_{0})$ solves
	\begin{equation}\label{eq:SPDE_skeleton}
		\partial_{t}\rho=\Delta\rho+\vdiv(\rho(\nabla W\ast\rho))+\sqrt{2}\str\vdiv(\rho h)\,,\qquad\rho|_{t=0}=\rho_{0}\,.
	\end{equation}
\end{theorem}
\begin{proof}
	It follows from Lemma~\ref{lem:continuity_flow_to_SPDE} that the
	map $\flow\to\rho(\place,\place;\flow,\rho_{0})$ is continuous from $C([0,T];\diffvol{2})$
	to $C([0,T];\init{0})$. Further, we know from Theorem~\ref{thm:support_theorem_flow},
	that
	\begin{equation*}
		\supp(\flow)=\cl\{\flow_{h}:h\in\cm(\colour)\}\qquad\text{in}\qquad C([0,T];C^{2}(\TT^{d};\TT^{d}))\,.
	\end{equation*}
	Hence, an application of Lemma~\ref{lem:support_push_forward} yields
	\begin{equation*}
		\supp(\rho(\place,\place;\flow,\rho_{0}))=\cl\{\rho(\place,\place;\flow,\rho_{0}):\flow\in\cl\{\flow_{h}:h\in\cm(\colour)\}\}\,.
	\end{equation*}
	Using once again that the map $\flow\mapsto\rho(\place,\place;\flow,\rho_{0})$ is continuous, we obtain
	\begin{equation*}
		\begin{split}
			&\cl\{\rho(\place,\place;\flow,\rho_{0}):\flow\in\cl\{\flow_{h}:h\in\cm(\colour)\}\}\\
			&=\cl\{\rho(\place,\place;\flow_{h},\rho_{0}):h\in\cm(\colour)\}\,.
		\end{split}
	\end{equation*}
	By the definition of $\flow_{h}$, we obtain that each $\rho(\place,\place;\flow_{h},\rho_{0})$ is a solution to the PDE~\eqref{eq:SPDE_skeleton},
	which yields the claim.
\end{proof}
\subsection{Existence as a random dynamical system}\label{subsec:RDS_for_SPDE}
Let $(\Omega,\mathcal{F}^{0},\PP,(\shift_{t})_{t\in\RR})$ be the
MDS constructed in Section~\ref{sec:prelim}. 
It then follows by Lemma~\ref{lem:RDS_for_stochastic_characteristics}
that the stochastic characteristics induced by~\eqref{eq:stochastic_characteristics_SDE}
generate an RDS~$\flow$ over $(\Omega,\mathcal{F}^{0},\PP,(\shift_{t})_{t\in\RR})$. In this subsection we use the continuity of the solution map of the SPDE~\eqref{eq:SPDE} with respect to the stochastic characteristics (Lemma~\ref{lem:continuity_flow_to_SPDE}) to argue that~\eqref{eq:SPDE} generates
an RDS $\rds\from[0,\infty)\times\Omega\times\init{0}\to\init{0}$
over the same MDS. 
\begin{theorem}\label{thm:RDS_for_SPDE}
	Let $W\in C^{2}(\TT^{d})$, $\reg>2$, $\colour\in h^{\reg}(\ZZ_{0}^{d})$ be radially symmetric and $\str>0$. Then there exists a unique (up to indistinguishability) continuous RDS $\rds\from[0,\infty)\times\Omega\times\init{0}\to\init{0}$ over $(\Omega,\mathcal{F}^{0},\PP,(\shift_{t})_{t\in\RR})$ which is a pathwise weak solution to the SPDE~\eqref{eq:SPDE} in the sense of Definition~\ref{def:weak_solution_SPDE}.
\end{theorem}
\begin{proof}
	Let $\flow$ be the RDS for the stochastic characteristics constructed in Lemma~\ref{lem:RDS_for_stochastic_characteristics}.
	For every $\omega\in\Omega$ and $\rho_{0}\in\init{0}$, denote
	by $(t,x)\mapsto\overline{\rho}(t,x;\flow(\omega),\rho_{0})$ the
	solution to the path-by-path PDE~\eqref{eq:path_by_path_PDE} driven by $\flow(\omega)\in C([0,\infty);\diffvol{2})$ and started from $\rho_{0}$ (for the well-posedness of the equation see Lemma~\ref{lem:well_posedness_flow_transformation}). Define
	\begin{equation*}
		\rds(t,\omega,\rho_{0})(x)\defeq\overline{\rho}(t,\flow^{-1}(t,\omega,x);\flow(\omega),\rho_{0})\,.
	\end{equation*}
	An application of Lemma~\ref{lem:well_posedness_SPDE} shows that~$\rds$ is the pathwise weak solution to~\eqref{eq:SPDE}.

	It follows by a tedious but elementary calculation that the  perfect cocycle property holds, that is, for all $s,t\in[0,\infty)$, $\omega\in\Omega$ and $\rho_0\in \init{0}$, it holds that
	\begin{equation*}
		\rds(t,\shift_{s}\omega,\rds(s,\omega,\rho_{0}))=\rds(t+s,\omega,\rho_{0})\,.
	\end{equation*}
	The $\bigl(\mathcal{B}([0,\infty))\otimes\mathcal{F}^{0}\otimes\mathcal{B}(\init{0}),\mathcal{B}(\init{0})\bigr)$-measurability
	of the map $P\from[0,\infty)\times\Omega\times \init{0}\to \init{0}$
	follows by Lemma~\ref{lem:joint_measurability_for_RDS} in Appendix~\ref{app:joint_measurability}.
\end{proof}
\section{Uniqueness of the invariant probability measure}
\label{sec:uniqueness_invariant_measure}
In this section we show for sufficiently large noise intensities $\str>0$ that the SPDE~\eqref{eq:SPDE} admits a unique invariant probability measure, which is given by the Dirac measure concentrated at the uniform state~$\unif$. To this end, we first show in Subsection~\ref{subsec:Lyapunov_exponents_stable_manifold} that the top Lyapunov exponent associated to the random dynamical system for~\eqref{eq:SPDE} is strictly negative, which yields the existence of a stable manifold around~$\unif$. We then argue in Subsection~\ref{subsec:reachability} that there exists a positive probability of reaching an arbitrary small $H^{-1}(\TT^{d})$-neighbourhood around~$\unif$ in finite time. In Subsection~\ref{subsec:regularisation} we then use the smoothing properties of the SPDE~\eqref{eq:SPDE} to deduce that any trajectory can be brought into an arbitrary small $L^{2}(\TT^{d})$-neighbourhood around~$\unif$ with positive probability, conditioned on the event that it was in a sufficiently small $H^{-1}(\TT^{d})$-neighbourhood around~$\unif$ at some previous time. Finally, we combine these properties in Subsection~\ref{subsec:uniqueness} to deduce the uniqueness of the invariant probability measure for~\eqref{eq:SPDE}.
\subsection{Lyapunov exponents and stable manifold theorem}\label{subsec:Lyapunov_exponents_stable_manifold}
Let $\rds\from[0,\infty)\times\Omega\times\init{0}\to\init{0}$ be the random dynamical system for the SPDE~\eqref{eq:SPDE} constructed in Theorem~\ref{thm:RDS_for_SPDE}. An application of Kingman's subadditive ergodic theorem yields the $\PP$-almost sure existence of the top Lyapunov exponent at the uniform state~$\unif$, that is,
\begin{equation}\label{eq:top_Lyapunov_exponent}
	\lambda_{\textnormal{top}}(\omega)\defeq\lim_{t\to\infty}\frac{1}{t}\log\norm{\D P(t,\omega,\unif)}_{L(L^{2}_{0}(\TT^{d}))}<\infty\,.
\end{equation}
In this subsection we show that there exists some $\lambda<0$ such that $\PP$-a.s.
\begin{equation*}
	\lambda_{\textnormal{top}}(\omega)<\lambda\,.
\end{equation*}
Let $\lin_{0}\in L^{2}_{0}(\TT^{d})$. An explicit calculation yields that 
\begin{equation*}
	\lin(t,x;\omega,\lin_{0})\defeq(\D\rds(t,\omega,\unif)\lin_{0})(x)
\end{equation*}
is given by the linearisation of~\eqref{eq:SPDE} around the uniform state $\unif$,
\begin{equation}\label{eq:linearisation}
	\begin{cases}
		\begin{aligned}
			\partial_{t}\lin&=\Delta \lin+\Delta W\ast \lin+\sqrt{2}\str\nabla \lin\circ\xi^{\colour}&&\quad t\in(0,T]\,,\, x\in\TT^{d}\\
			\lin\tzero&=\lin_{0}&&\quad x\in\TT^{d}
		\end{aligned}
	\end{cases}
	\,.
\end{equation}
Let $\eta\in(0,1)$ and recall from~\eqref{eq:maximal_growth_rate} (resp.~\eqref{eq:dimension_constant}) the definition of the constant~$C_{W}^{(\eta)}$ (resp.~$C_{d}$). 
\begin{theorem}\label{thm:L2_decay_almost_surely_energy_spectrum} 
	Let $W\in L^{1}(\TT^{d})$, $\eta\in(0,1)$ and $\colour\in\ell^{2}(\ZZ_{0}^{d})$ be radially symmetric. Assume that $\str>0$ satisfies
	\begin{equation*}
		C_{W}^{(\eta)}-(1-\eta)<\norm{\colour}_{h^{-1}(\ZZ^{d}_{0})}^{2}C_{d}\str^2\,.
	\end{equation*}
	Then for all $\gamma>0$ such that
	\begin{equation*}
		0<\gamma<-C_{W}^{(\eta)}+(1-\eta)+\norm{\colour}_{h^{-1}(\ZZ^{d}_{0})}^{2}C_{d}\str^2\,,
	\end{equation*}
	there exists a random constant $D=D(\omega)$ depending on $d$, $\eta$, $\str$, $\gamma$, $C_{W}^{(\eta)}$ and $\norm{\colour}_{h^{-1}(\ZZ^{d}_{0})}$,
	such that for any $\lin_{0}\in L^{2}_{0}(\mathbb{T}^{d})$, 
	\begin{equation*}
		\lVert\lin(t,\place;\omega,\lin_{0})\rVert_{L^{2}(\TT^{d})}^{2}\leq D(\omega)\euler^{-2(2\uppi)^2\gamma t}\lVert \lin_{0}\rVert_{L^{2}(\TT^{d})}^{2}\,,\qquad\mathbb{P}\text{-a.s.\ in }\omega\in\Omega\,.
	\end{equation*}
\end{theorem}
\begin{proof}
	Upon adding the lower-order term given by the convolution against $\Delta W$, it follows by the techniques of~\cite[Thm.~1.3]{luo_tang_zhao_24}, that there exists some $C>0$ such that,
	\begin{equation*}
		\begin{split}
			&\mathbb{E}[\lVert\lin(t,\place;\place,\lin_{0})\rVert_{L^{2}(\TT^{d})}^{2}]\leq C\frac{(1-\eta)+\norm{\colour}_{h^{-1}(\ZZ^{d}_{0})}^{2}C_{d}\str^{2}}{1-\eta}\\
			&\quad\times\exp\Bigl(2(2\uppi)^{2}\Bigl(C_{W}^{(\eta)}-(1-\eta)-\norm{\colour}_{h^{-1}(\ZZ^{d}_{0})}^{2}C_{d}\str^{2}\Bigr)t\Bigr)\lVert \lin_{0}\rVert_{L^{2}(\TT^{d})}^{2}\,.
		\end{split}
	\end{equation*}
	A Borel--Cantelli argument then allows us to conclude, for details see~\cite[Thm.~1.3]{luo_tang_zhao_24}.
\end{proof}
\begin{theorem}\label{thm:Lyapunov_exponent_negative_energy_spectrum}
	Let $W\in L^{1}(\TT^{d})$, $\eta\in(0,1)$ and $\colour\in \ell^{2}(\ZZ_{0}^{d})$ be radially symmetric. Assume that $\str>0$ satisfies
	\begin{equation*}
		C_{W}^{(\eta)}-(1-\eta)<\norm{\colour}_{h^{-1}(\ZZ^{d}_{0})}^{2}C_{d}\str^{2}
	\end{equation*}
	and let $\gamma$ be such that 
	\begin{equation*}
		0<\gamma<-C_{W}^{(\eta)}+(1-\eta)+\norm{\colour}_{h^{-1}(\ZZ^{d}_{0})}^{2}C_{d}\str^{2}\,,
	\end{equation*}
	where $C_{W}^{(\eta)}$ (resp.~$C_{d}$) is the constant defined in~\eqref{eq:maximal_growth_rate} (resp.~\eqref{eq:dimension_constant}). For $\lambda\defeq-(2\uppi)^{2}\gamma$, it then follows that,
	\begin{equation*}
		\limsup_{t\to\infty}\frac{1}{t}\log\norm{\D\rds(t,\omega,\unif)}_{L(L^{2}_{0}(\TT^{d}))}<\lambda\,,\qquad\mathbb{P}\text{-a.s.\ in }\omega\in\Omega\,.
	\end{equation*}
\end{theorem}
\begin{proof}
	For every $\lin_0\in L^{2}_{0}(\TT^{d})$, let $(\omega,t,x)\mapsto \lin(t,x;\omega,\lin_0)$ be the solution to~\eqref{eq:linearisation}. An application of Theorem~\ref{thm:L2_decay_almost_surely_energy_spectrum}
	yields the existence of a random constant~$D$ independent of~$\lin_{0}$, such that $\PP$-a.s.\ in $\omega\in\Omega$,
	\begin{equation*}
		\lVert \lin(t,\place;\omega,\lin_0)\rVert_{L^{2}(\TT^{d})}^{2}\leq D(\omega)\euler^{-2(2\uppi)^{2}\gamma t}\lVert \lin_{0}\rVert_{L^{2}(\TT^{d})}^{2}\,,
	\end{equation*}
	which yields
	\begin{equation*}
		\begin{split}
			\norm{\D\rds(t,\omega,\unif)}_{L(L^{2}_{0}(\TT^{d}))}&=\sup_{\substack{\lin_{0}\in L^{2}_{0}(\TT^{d})\\\norm{\lin_{0}}_{L^{2}(\TT^{d})}=1}}\lVert \lin(t,\place;\omega,\lin_{0})\rVert_{L^{2}(\TT^{d})}\\
			&\leq D(\omega)^{1/2}\euler^{-(2\uppi)^{2}\gamma t}\,.
		\end{split}
	\end{equation*}
	Taking the logarithm, we obtain
	\begin{equation*}
		\log\norm{\D\rds(t,\omega,\unif)}_{L(L^{2}_{0}(\TT^{d}))}\leq\frac{1}{2}\log D(\omega)-(2\uppi)^{2}\gamma t\,,\qquad\mathbb{P}\text{-a.s.\ in }\omega\in\Omega\,.
	\end{equation*}
	which yields the claim upon dividing by~$t$ and taking the limit superior as $t\to\infty$.
\end{proof}
Having established the negativity of the top Lyapunov exponent~\eqref{eq:top_Lyapunov_exponent}, we can now apply the stable manifold theorem to deduce asymptotic stability.
\begin{theorem}\label{thm:asymptotic_stability}
	Let $W\in C^{2}(\TT^{d})$, $\eta\in(0,1)$, $\reg>2$, $\colour\in h^{\reg}(\ZZ_{0}^{d})$ be radially symmetric and $\str>0$ be such that
	\begin{equation*}
		C_{W}^{(\eta)}-(1-\eta)<\norm{\colour}_{h^{-1}(\ZZ^{d}_{0})}^{2}C_{d}\str^{2}\,.
	\end{equation*}
	Then there exists a (deterministic) $\delta>0$ such that
	\begin{equation*}
		\PP\Bigl(\lim_{n\to\infty}\diam\bigl(\rds(n,\place,B^{0}_{\delta}(\unif))\bigr)=0\Bigr)>0\,.
	\end{equation*}
\end{theorem}
\begin{proof}
	An application of Theorem~\ref{thm:RDS_for_SPDE} yields that $\rds\from[0,\infty)\times\Omega\times\init{0}\to\init{0}$
	is a random dynamical system over the metric dynamical system $(\Omega,\mathcal{F}^{0},\PP,(\shift_{t})_{t\in\RR})$.
	Further, we have shown in Theorem~\ref{thm:Lyapunov_exponent_negative_energy_spectrum},
	that~$\rds$ has a negative top Lyapunov exponent at the constant
	state~$\unif$. The claim then follows by the stable
	manifold theorem on separable Hilbert spaces and a covering argument, see~\cite[Cor.~1]{scheutzow_vorkastner_18} and~\cite[Proof of Lem.~3.3]{flandoli_gess_scheutzow_17}.
\end{proof}
\subsection{Reachability}\label{subsec:reachability}
In this subsection we first show for any $\delta>0$ and $A\subset\init{1}$ bounded, that there exists a time $t_{1}=t_{1}(\delta,A)$ such that the solution to the SPDE~\eqref{eq:SPDE} started from $\rho_{0}\in A$ reaches the $\dot{H}^{-1}(\TT^{d})$-ball of radius~$\delta$ around the uniform state~$\unif$ with positive probability at time~$t_{1}$ (Lemma~\ref{lem:reachability}). We then use a lower-semicontinuity argument to ensure that this positive probability can be bounded away from zero uniformly over all initial conditions in $A$ (Theorem~\ref{thm:reachability_uniform}).
\begin{lemma}\label{lem:reachability}
	Let $W\in C^{2}(\TT^{d})$, $\reg>4$, $\colour\in h^{\reg}(\ZZ_{0}^{d})$ be non-trivial and radially symmetric, $\str>0$, $\delta>0$ and $A\subset\init{1}$ be bounded.
	Then there exists some $t_{1}=t_{1}(\delta,A)\geq0$, which also depends on $\str$, $\colour$ and $W$,
	such that for all $\rho_{0}\in A$,
	\begin{equation*}
		\mathbb{P}\bigl(\lVert\rds(t_{1},\place,\rho_{0})-\unif\rVert_{\dot{H}^{-1}(\T^{d})}<\delta\bigr)>0\,.
	\end{equation*}
\end{lemma}
\begin{proof}
	Let $\delta>0$ and $A\subset\init{1}$ be bounded.
	Denote $A-\unif\defeq\{\rho_{0}-\unif:\rho_{0}\in A\}$.
	By Lemma~\ref{lem:existence_Cameron_Martin_mixer}
	there exist some $\tilde{t}_{1}=\tilde{t}_{1}(\delta,A-\unif,\str,\colour)$ and $v=v(\delta,A-\unif,\str,\colour)\in\cm(\colour)$
	such that for any $u_{0}\defeq\rho_{0}-\unif \in A-\unif$ and any solution $(t,x)\mapsto u_{v}(t,x;u_{0})$ to the transport
	equation 
	\begin{equation*}
		\partial_{t}u=\sqrt{2}\str\nabla\cdot(uv)\,,\qquad u\tzero=\rho_{0}\,,
	\end{equation*}
	it holds that uniformly in $\rho_{0}\in A$,
	\begin{equation}\label{eq:mixing_estimate_application}
		\lVert u_{v}(\tilde{t}_{1},\place;\rho_{0})-\unif\rVert_{\dot{H}^{-1}(\TT^{d})}=\lVert u_{v}(\tilde{t}_{1},\place;u_{0})\rVert_{\dot{H}^{-1}(\TT^{d})}\leq\delta/2\,.
	\end{equation}
	Let $\scale>0$ and $(t,x)\mapsto\rho_{v}^{(\scale)}(t,x;\rho_{0})$ be the solution to the
	following PDE 
	\begin{equation*}
		\partial_{t}\rho=\Delta\rho+\nabla\cdot(\rho(\nabla W\ast\rho))+\scale^{-1}\sqrt{2}\str\nabla\cdot(\rho v)\,,\qquad\rho\tzero=\rho_{0}\,.
	\end{equation*}
	Rescaling in time, we define $\tilde{\rho}_{v}^{(\scale)}(t,x;\rho_{0})\defeq\rho_{v}^{(\scale)}(\scale t,x;\rho_{0})$
	for every $t\geq0$ and $x\in\TT^{d}$. We then have that $\tilde{\rho}_{v}^{(\scale)}$
	solves
	\begin{equation}\label{eq:skeleton_strong_noise_time_change}
		\partial_{t}\tilde{\rho}=\scale(\Delta\tilde{\rho}+\nabla\cdot(\tilde{\rho}(\nabla W\ast\tilde{\rho})))+\sqrt{2}\str\nabla\cdot(\tilde{\rho}v)\,,\qquad\tilde{\rho}\tzero=\rho_{0}\,.
	\end{equation}
	Let us now define for every $t\geq0$,
	\begin{equation*}
		\begin{split}
			S_{t}\from\R_{>0}\times \init{1} & \to H^{-1}(\TT^{d})\,,\\
			(\scale,\rho_{0}) & \mapsto S_{t}(\scale;\rho_{0})\defeq\tilde{\rho}_{v}^{(\scale)}(t,\place;\rho_{0})\,,
		\end{split}
	\end{equation*}
	to be the solution map associated to~\eqref{eq:skeleton_strong_noise_time_change}.
	We can then decompose
	\begin{equation}\label{eq:recurrence_decomposition}
		S_{t}(\scale;\rho_{0})-\unif=\bigl(S_{t}(0;\rho_{0})-\unif\bigr)+\bigl(S_{t}(\scale;\rho_{0})-S_{t}(0;\rho_{0})\bigr)
	\end{equation}
	and control each term separately.

	Using that $S_{t}(0;\rho_{0})-\unif=u_{v}(t,\place;u_{0})$, we obtain by~\eqref{eq:mixing_estimate_application} that uniformly in $\rho_{0}\in A$,
	\begin{equation*}
		\lVert S_{\tilde{t}_{1}}(0;\rho_{0})-\unif\rVert_{\dot{H}^{-1}(\T^{d})}\leq\delta/2\,.
	\end{equation*}

	Further, an application of Lemma~\ref{lem:zeroth_order_bound}, yields
	that there exists some $\scale_{0}=\scale_{0}(\delta,A,\str,\colour,W)$ such that for every
	$\scale\leq\scale_{0}$, uniformly in $\rho_{0}\in A$,
	\begin{equation*}
		\norm{S_{\tilde{t}_{1}}(\scale;\rho_{0})-S_{\tilde{t}_{1}}(0;\rho_{0})}_{\dot{H}^{-1}(\TT^{d})}\leq\delta/2\,.
	\end{equation*}
	Combining the above with~\eqref{eq:recurrence_decomposition}, we
	obtain that there exists some $\tilde{t}_{1}=\tilde{t}_{1}(\delta,A-\unif,\str,\colour)$
	and $v=v(\delta,A-\unif,\str,\colour)\in\cm(\colour)$ such that for every $\scale\leq\scale_{0}$, uniformly in $\rho_{0}\in A$,
	\begin{equation*}
		\lVert S_{\tilde{t}_{1}}(\scale;\rho_{0})-\unif\rVert_{\dot{H}^{-1}(\TT^{d})}\leq\delta\,.
	\end{equation*}
	Define $t_{1}\defeq\scale\tilde{t}_{1}$ and note that $t_{1}=t_{1}(\delta,A,\str,\colour,W)$
	is uniform in $\rho_{0}\in A$ since $\tilde{t}_{1}=\tilde{t}_{1}(\delta,A-\unif,\str,\colour)$
	and $\scale=\scale(\delta,A,\str,\colour,W)$ are both uniform in $\rho_{0}\in A$.
	Using that
	\begin{equation*}
		S_{\tilde{t}_{1}}(\scale;\rho_{0})=\rho_{v}^{(\scale)}(\scale\tilde{t}_{1},\place;\rho_{0})=\rho_{v}^{(\scale)}(t_{1},\place;\rho_{0})\,,
	\end{equation*}
	we obtain
	\begin{equation*}
		\lVert\rho_{v}^{(\scale)}(t_{1},\place;\rho_{0})-\unif\rVert_{\dot{H}^{-1}(\TT^{d})}\leq\delta\,.
	\end{equation*}
	Let
	\begin{equation*}
		N_{\rho_{0},t_{1},\delta}\defeq\Bigl\{\rho\in C([0,T];\init{0}):\rho(0)=\rho_{0},\,\lVert\rho(t_{1})-\unif\rVert_{\dot{H}^{-1}(\TT^{d})}<\delta\Bigr\}\,.
	\end{equation*}
	In particular, $N_{\rho_{0},t_{1},\delta}$ is a neighbourhood of $\rho_{v}^{(\alpha)}(\place,\place;\rho_{0})$ and we know by the support theorem (Theorem~\ref{thm:support_theorem_McKean_Vlasov}),
	that
	\begin{equation*}
		\rho_{v}^{(\scale)}(\place,\place;\rho_{0})\in\supp(\rds(\place,\place,\rho_{0}))\,.
	\end{equation*}
	By the definition of the support of a measure, it follows that the
	every open neighbourhood of $\rho_{v}^{(\scale)}$ has positive measure with respect to the law of $\rds(\place,\place,\rho_{0})$,
	hence in particular $\PP(\rds(\place,\place,\rho_{0})\in N_{\rho_{0},t_{1},\delta})>0$. It follows that
	for every $\rho_{0}\in A$,
	\begin{equation*}
		\mathbb{P}\bigl(\lVert\rds(t_{1},\place,\rho_{0})-\unif\rVert_{\dot{H}^{-1}(\TT^{d})}<\delta\bigr)=\PP\bigl(\rds(\place,\place,\rho_{0})\in N_{\rho_{0},t_{1},\delta}\bigr)>0\,,
	\end{equation*}
	which yields the claim.
\end{proof}
Next we use a lower-semicontinuity argument to show that the probability of reaching a neighbourhood around the uniform state~$\unif$ is positive uniformly over all initial conditions from a closed and bounded subset $A\subset\init{1}$.
\begin{lemma}\label{lem:reachability_lsc}
	Let $W\in C^{2}(\TT^{d})$, $\reg>2$, $\colour\in h^{\reg}(\ZZ_{0}^{d})$ be radially symmetric, $\str>0$, $\delta>0$ and $t\geq0$. It then
	follows that the map 
	\begin{equation*}
		\init{0}\ni\rho_{0}\mapsto\mathbb{P}\bigl(\lVert\rds(t,\place,\rho_{0})-\unif\rVert_{\dot{H}^{-1}(\T^{d})}<\delta\bigr)
	\end{equation*}
	is lower semicontinuous.
\end{lemma}
\begin{proof}
	Let $(\rho_{0}^{(k)})_{k\in\mathbb{N}}$ be a sequence such that $\rho_{0}^{(k)}\in \init{0}$
	for each $k\in\mathbb{N}$ and $\rho_{0}^{(k)}\to\rho_{0}$ in $\init{0}$
	as $k\to\infty$. We obtain by Lemma~\ref{lem:weak_continuity}, the Portmanteau theorem and by the
	openness of the set 
	\begin{equation*}
		\{\rho\in H^{-1}(\mathbb{T}^{d}):\lVert\rho-\unif\rVert_{\dot{H}^{-1}(\TT^{d})}<\delta\}\,,
	\end{equation*}
	that
	\begin{equation*}
		\liminf_{k\to\infty}\mathbb{P}\bigl(\lVert\rds(t,\place,\rho_{0}^{(k)})-\unif\rVert_{\dot{H}^{-1}(\TT^{d})}<\delta\bigr)
		\geq\mathbb{P}\bigl(\lVert\rds(t,\place,\rho_{0})-\unif\rVert_{\dot{H}^{-1}(\TT^{d})}<\delta\bigr)\,,
	\end{equation*}
	which is equivalent to the lower semicontinuity of 
	\begin{equation*}
		\rho_{0}\mapsto\mathbb{P}\bigl(\lVert\rds(t,\place,\rho_{0})-\unif\rVert_{\dot{H}^{-1}(\TT^{d})}<\delta\bigr)\,.
	\end{equation*}
	This yields the claim.
\end{proof}
\begin{theorem}\label{thm:reachability_uniform}
	Let $W\in C^{2}(\TT^{d})$, $\reg>4$, $\colour\in h^{\reg}(\ZZ_{0}^{d})$ be non-trivial and radially symmetric, $\str>0$, $\delta>0$ and $A\subset \init{1}$ be closed and bounded. Then there exists some $t_{1}=t_{1}(\delta,A)\geq0$ and some constant $C(\delta,A)>0$, which also depend on $\str$, $\colour$ and $W$, such that 
	\begin{equation*}
		\inf_{\rho_{0}\in A}\mathbb{P}\bigl(\lVert\rds(t_{1},\place,\rho_{0})-\unif\rVert_{\dot{H}^{-1}(\T^{d})}<\delta\bigr)\geq C(\delta,A)\,.
	\end{equation*}
\end{theorem}
\begin{proof}
	Let $\delta>0$, $A\subset \init{1}$ be closed and bounded, and $t_{1}=t_{1}(\delta,A)$ be given by Lemma~\ref{lem:reachability}.
	The compact embedding $H^{1}(\TT^{d})\embed L^{2}(\TT^{d})$ yields that $A$ is a compact subset of $\init{0}$. 
	Lemma~\ref{lem:reachability_lsc} combined with the property that lower-semicontinuous functions attain their minima on compact sets, allows us to conclude that there exists some $\rho_{0}^{*}\in A$, such that
	\begin{equation*}
		\inf_{\rho_{0}\in A}\mathbb{P}\bigl(\lVert\rds(t_{1},\place,\rho_{0})-\unif\rVert_{\dot{H}^{-1}(\T^{d})}<\delta\bigr)=\mathbb{P}\bigl(\lVert\rds(t_{1},\place,\rho_{0}^{*})-\unif\rVert_{\dot{H}^{-1}(\T^{d})}<\delta\bigr)\,.
	\end{equation*}
	Lemma~\ref{lem:reachability} implies 
	\begin{equation*}
		C(\delta,A)\defeq\mathbb{P}\bigl(\lVert\rds(t_{1},\place,\rho_{0}^{*})-\unif\rVert_{\dot{H}^{-1}(\T^{d})}<\delta\bigr)>0\,,
	\end{equation*}
	which yields the claim.
\end{proof}
\subsection{Regularisation}
\label{subsec:regularisation}
In this subsection we show that for any $\delta>0$ there exists some $\delta'>0$ and $t_{2}>0$ such that any solution to the SPDE~\eqref{eq:SPDE} that lies in the $\dot{H}^{-1}(\TT^{d})$-ball of radius $\delta'$ around the uniform state~$\unif$ at some time $t_{1}$, also reaches the $L^{2}(\TT^{d})$-ball of radius $\delta$ around $\unif$ at time $t_{1}+t_{2}$ with probability at least $1/2$ (Lemma~\ref{lem:dynamics_regularisation}).
\begin{lemma}\label{lem:dynamics_regularisation}
	Let $W\in C^{2}(\mathbb{T}^{d})$, $\reg>2$, $\colour\in h^{\reg}(\ZZ_{0}^{d})$ be radially symmetric, $\str>0$ and $\delta>0$. Then there exists some $\delta'>0$ and $t_{2}>0$ such that for all $t_{1}\geq0$ and any $\rho_{0}\in \init{1}$, it holds that
	\begin{equation*}
		\begin{split}
			&\mathbb{P}\bigl(\rds(t_{1},\place,\rho_{0})\in B_{\delta'}^{-1}(\unif)\text{ and }\rds(t_{1}+t_{2},\place,\rho_{0})\in B_{\delta}^{0}(\unif)\bigr)\\
			&>\frac{1}{2}\mathbb{P}\bigl(\rds(t_{1},\place,\rho_{0})\in B_{\delta'}^{-1}(\unif)\bigr)\,.
		\end{split}
	\end{equation*}
\end{lemma}
\begin{proof}
	Let $\delta>0$. We first show that there exists some $\delta'>0$ and $t_{2}>0$ such that uniformly over $\rho_{1}\in B_{\delta'}^{-1}(\unif)$,
	\begin{equation}\label{eq:regularisation_uniform}
		\PP\bigl(\rds(t_{2},\place,\rho_{1})\in B_{\delta}^{0}(\unif)\bigr)>1/2\,.
	\end{equation}
	The claim then follows by a conditioning argument.

	For every $t>0$, $\delta_{0}>0$, $\delta'>0$ and $\rho_{1}\in B_{\delta'}^{-1}(\unif)$, it follows by an application of Markov's inequality, that
	\begin{equation*}
		\mathbb{P}\bigl(\rds(t,\place,\rho_{1})\not\in B_{\delta_{0}}^{0}(\unif)\bigr)\leq\frac{\mathbb{E}[\lVert\rds(t,\place,\rho_{1})-\unif\rVert_{L^{2}(\TT^{d})}^{2}]}{\delta_{0}^{2}}\,.
	\end{equation*}
	Let $T_{2}>T_{1}>0$. Integrating the above inequality over $t\in[T_{1},T_{2}]$, we obtain
	\begin{equation*}
		\int_{T_{1}}^{T_{2}}\mathbb{P}\bigl(\rds(t,\place,\rho_{1})\not\in B_{\delta_{0}}^{0}(\unif)\bigr)\,\dd t\leq\frac{1}{\delta_{0}^{2}}\int_{T_{1}}^{T_{2}}\mathbb{E}[\lVert\rds(t,\place,\rho_{1})-\unif\rVert_{L^{2}(\TT^{d})}^{2}]\,\dd t\,.
	\end{equation*}
	An application of Lemma~\ref{lem:H-1_bound} yields that
	there exists some constant
	\begin{equation*}
		C=C(d,\norm{W}_{C^{2}(\TT^{d})},\norm{\colour}_{h^{\reg}(\ZZ_{0}^{d})},\str)>0
	\end{equation*}
	independent of $\rho_{1}$ such that
	\begin{equation*}
		\int_{T_{1}}^{T_{2}}\mathbb{E}\bigl[\lVert\rds(t,\place,\rho_{1})-\unif\rVert_{L^{2}(\TT^{d})}^{2}\bigr]\,\dd t\leq\euler^{CT_{2}}\EE[\lVert\rho_{1}-\unif\rVert_{\dot{H}^{-1}(\TT^{d})}^{2}]\,,
	\end{equation*}
	which implies
	\begin{equation}\label{eq:regularisation_integrated_bound}
		\int_{T_{1}}^{T_{2}}\mathbb{P}\bigl(\rds(t,\place,\rho_{1})\not\in B_{\delta_{0}}^{0}(\unif)\bigr)\,\dd t\leq\euler^{CT_{2}}\frac{\EE\bigl[\lVert\rho_{1}-\unif\rVert_{\dot{H}^{-1}(\TT^{d})}^{2}\bigr]}{\delta_{0}^{2}}\leq \euler^{CT_{2}}\frac{(\delta')^{2}}{\delta_{0}^{2}}\,.
	\end{equation}
	Choosing $\delta'$ sufficiently small depending on $T_{1}$, $T_{2}$, $\delta_{0}$ and $C$, we obtain
	\begin{equation*}
		\euler^{CT_{2}}\frac{(\delta')^{2}}{\delta_{0}^{2}}<\frac{T_{2}-T_{1}}{4}\,.
	\end{equation*}
	Define for every $\rho_{1}\in B_{\delta'}^{-1}(\unif)$,
	\begin{equation*}
		\mathcal{T}_{\rho_{1}}\defeq\Bigl\{t\in[T_{1},T_{2}]:\mathbb{P}\bigl(\rds(t,\place,\rho_{1})\not\in B_{\delta_{0}}^{0}(\unif)\bigr)<1/2\Bigr\}\,.
	\end{equation*}
	Assume that $\lvert\mathcal{T}_{\rho_{1}}\rvert<(T_{2}-T_{1})/2$. Then,
	\begin{equation*}
		\begin{split}
			\int_{T_{1}}^{T_{2}}\mathbb{P}\bigl(\rds(t,\place,\rho_{1})\not\in B_{\delta_{0}}^{0}(\unif)\bigr)\,\dd t &\geq\int_{\mathcal{T}_{\rho_{1}}^\complement}\mathbb{P}\bigl(\rds(t,\place,\rho_{1})\not\in B_{\delta_{0}}^{0}(\unif)\bigr)\,\dd t\\
			&>\frac{T_{2}-T_{1}}{4}>\euler^{CT_{2}}\frac{(\delta')^{2}}{\delta_{0}^{2}}\,,
		\end{split}
	\end{equation*}
	which contradicts~\eqref{eq:regularisation_integrated_bound}. Hence, $\lvert\mathcal{T}_{\rho_{1}}\rvert\geq(T_{2}-T_{1})/2$. 
	
	Consequently, for each $\delta_{0}>0$ and $T_{2}>T_{1}>0$ there exists some $\delta'>0$ such that for every $\rho_{1}\in B_{\delta'}^{-1}(\unif)$ we can find some $\tau_{\rho_{1}}\in[T_{1},T_{2}]$ such that
	\begin{equation*}
		\mathbb{P}\bigl(\rds(\tau_{\rho_{1}},\place,\rho_{1})\in B_{\delta_{0}}^{0}(\unif)\bigr)>1/2\,.
	\end{equation*}

	The cocycle property combined with a pathwise energy estimate based on Nash's inequality yields that there exists some constant $C_{*}=C_{*}(d,\norm W_{C^{2}(\TT^{d})})$ such that for each $\rho_{1}\in B_{\delta'}^{-1}(\unif)$, $\PP$-a.s.\ in $\omega\in\Omega$,
	\begin{equation*}
		\norm{\rds(T_{2},\omega,\rho_{1})-\unif}_{L^{2}(\TT^{d})}^{2}\leq\norm{\rds(\tau_{\rho_{1}},\omega,\rho_{1})-\unif}_{L^{2}(\TT^{d})}^{2}+C_{*}(T_{2}-\tau_{\rho_{1}})\,.
	\end{equation*}
	For every $\delta>0$, we can choose $T_{2}>T_{1}>0$ such that
	\begin{equation*}
		C_{*}(T_{2}-\tau_{\rho_{1}})\leq C_{*}(T_{2}-T_{1})<\delta^{2}/2
	\end{equation*}
	and $\delta_{0}>0$ such that $\delta_{0}^{2}<\delta^{2}/2$.
	Therefore, for every $\delta>0$ there exists some $\delta'>0$ such that for every $\rho_{1}\in B_{\delta'}^{-1}(\unif)$,
	\begin{equation*}
		\begin{split}
			\PP\bigl(\rds(T_{2},\place,\rho_{1})\in B_{\delta}^{0}(\unif)\bigr)&=\PP\bigl(\norm{\rds(T_{2},\place,\rho_{1})-\unif}_{L^{2}(\TT^{d})}<\delta\bigr)\\
			&\geq\PP\bigl(\norm{\rds(\tau_{\rho_{1}},\place,\rho_{1})-\unif}_{L^{2}(\TT^{d})}<\delta_{0}\bigr)\\
			&>1/2\,.
		\end{split}
	\end{equation*}
	Choosing $t_{2}\defeq T_{2}$, we obtain~\eqref{eq:regularisation_uniform}.
	
	Let $\delta'>0$ and $t_{2}>0$ be as above. For all $\rho_{0}\in\init{1}$ and $t_{1}>0$ it follows by an application of the cocycle property, that
	\begin{equation*}
		\rds(t_{1}+t_{2},\place,\rho_{0})=\rds(t_{2},\shift_{t_{1}}\place,\rds(t_{1},\place,\rho_{0}))\,.
	\end{equation*}
	Note that $\rds(t_{1},\place,\rho_{0})$ is $(\mathcal{F}_{0,t_{1}},\mathcal{B}(\init{0}))$-measurable and that $\rds(t_{2},\shift_{t_{1}}\place,\rho_{1})$ is $(\mathcal{F}_{t_{1},t_{1}+t_{2}},\mathcal{B}(\init{0}))$-measurable for fixed $\rho_{1}$.
	Conditioning on the value of $P(t_{1},\place,\rho_{0})$, it follows by the independence of $\mathcal{F}_{0,t_{1}}$ and $\mathcal{F}_{t_{1},t_{1}+t_{2}}$, that
	\begin{equation*}
		\begin{split}
			&\mathbb{P}\bigl(\rds(t_{1},\place,\rho_{0})\in B_{\delta'}^{-1}(\unif),\,\rds(t_{1}+t_{2},\place,\rho_{0})\in B_{\delta}^{0}(\unif)\bigr)\\
			&=\int_{B_{\delta'}^{-1}(\unif)}\mathbb{P}\bigl(\rds(t_{2},\shift_{t_{1}}\place,\rho_{1})\in B_{\delta}^{0}(\unif)\bigr)\,\dd\law\bigl(\rds(t_{1},\place,\rho_{0})\bigr)(\rho_{1})\,.
		\end{split}
	\end{equation*}
	Using the $(\shift_{t})_{t\geq0}$-invariance of $\PP$, we then obtain
	\begin{equation}\label{eq:regularisation_conditioning}
		\begin{split}
			&\mathbb{P}\bigl(\rds(t_{1},\place,\rho_{0})\in B_{\delta'}^{-1}(\unif),\,\rds(t_{1}+t_{2},\place,\rho_{0})\in B_{\delta}^{0}(\unif)\bigr)\\
			&=\int_{B_{\delta'}^{-1}(\unif)}\mathbb{P}\bigl(\rds(t_{2},\place,\rho_{1})\in B_{\delta}^{0}(\unif)\bigr)\,\dd\law\bigl(\rds(t_{1},\place,\rho_{0})\bigr)(\rho_{1})\,.
		\end{split}
	\end{equation}
	It then follows by~\eqref{eq:regularisation_uniform}~\&~\eqref{eq:regularisation_conditioning}, that
	\begin{equation*}
		\begin{split}
			&\mathbb{P}\bigl(\rds(t_{1},\place,\rho_{0})\in B_{\delta'}^{-1}(\unif),\,\rds(t_{1}+t_{2},\place,\rho_{0})\in B_{\delta}^{0}(\unif)\bigr)\\
			&>\frac{1}{2}\mathbb{P}\bigl(\rds(t_{1},\place,\rho_{0})\in B_{\delta'}^{-1}(\unif)\bigr)\,.
		\end{split}
	\end{equation*}
	This yields the claim.
\end{proof}
\subsection{Proof of the main result}
\label{subsec:uniqueness}
In this subsection we combine asymptotic stability (Theorem~\ref{thm:asymptotic_stability}) with reachability (Theorem~\ref{thm:reachability_uniform}) and uniform smoothing (Lemma~\ref{lem:dynamics_regularisation}) to prove that the SPDE~\eqref{eq:SPDE} started from initial data~$\rho_{0}$ enters an $L^{2}(\TT^{d})$-ball of radius~$\eps$ around the uniform state~$\unif$ with positive probability, asymptotically as time tends to infinity, uniformly over $\eps>0$ and $\rho_{0}$ in a bounded and closed subset of $\init{1}$ (Lemma~\ref{lem:key_reachability}). This is the key step in the proof of our main result (Theorem~\ref{thm:main_result}) which we present at the end of this subsection.
\begin{lemma}\label{lem:key_reachability}
	Let $W\in C^{2}(\TT^{d})$, $\eta\in(0,1)$, $\reg>4$, $\colour\in h^{\reg}(\ZZ_{0}^{d})$ be non-trivial and radially symmetric, $\str>0$ such that
	\begin{equation*}
		C_{W}^{(\eta)}-(1-\eta)<\norm{\colour}_{h^{-1}(\ZZ^{d}_{0})}^{2}C_{d}\str^{2}
	\end{equation*}
	and $A\subset\init{1}$ be closed and bounded.
	Then there exists some $c>0$ such that uniformly in $\rho_{0}\in A$ and $\eps>0$, 
	\begin{equation*}
		\liminf_{t\to\infty}\mathbb{P}\bigl(\rds(t,\place,\rho_{0})\in B_{\eps}^{0}(\unif)\bigr)\geq c\,.
	\end{equation*}
\end{lemma}
\begin{proof}
	Let~$\delta$ be as in Theorem~\ref{thm:asymptotic_stability}, $\delta'>0$ and~$t_{2}$ be as in Lemma~\ref{lem:dynamics_regularisation} and $t_{1}=t_{1}(\delta',A)$ be as in Theorem~\ref{thm:reachability_uniform}. We proceed in several steps.

	\emph{Step 1 (asymptotic stability):}
	An application of Theorem~\ref{thm:asymptotic_stability} yields
	\begin{equation*}
		C(\delta)\defeq\PP\bigl(\lim_{n\to\infty}\diam(\rds(n,\place,B^{0}_{\delta}(\unif)))=0\bigr)>0\,.
	\end{equation*}
	In particular it follows that, for every $\eps>0$,
	\begin{equation*}
		\begin{split}
			C(\delta)&\leq\mathbb{P}\bigl(\limsup_{n\to\infty}\diam\rds(n,\place,B_{\delta}^{0}(\unif))<\eps\bigr)\\
			&=\mathbb{P}\bigl(\exists n\in\NN\,\forall j\geq n\,\diam\rds(j,\place,B_{\delta}^{0}(\unif))<\eps\bigr)\\
			&=\mathbb{P}\bigl(\liminf_{n\to\infty}\{\diam\rds(n,\place,B_{\delta}^{0}(\unif))<\eps\}\bigr)\\
			&\leq\liminf_{n\to\infty}\mathbb{P}\bigl(\diam\rds(n,\place,B_{\delta}^{0}(\unif))<\eps\bigr)\\
			&\leq\liminf_{n\to\infty}\mathbb{P}\bigl(\rds(n,\place,B_{\delta}^{0}(\unif))\subset B_{\eps}^{0}(\unif)\bigr)\,.
		\end{split}
	\end{equation*}
	Hence, for all $\eps>0$ there exists some $N=N(\eps)\in\mathbb{N}$ such that uniformly in $n\geq N$,
	\begin{equation}\label{eq:step_asymptotic_stability}
		\mathbb{P}\bigl(\rds(n,\place,B_{\delta}^{0}(\unif))\subset B_{\eps}^{0}(\unif)\bigr)>\frac{1}{2}C(\delta)\,.
	\end{equation}

	\emph{Step 2 (reachability):} 
	It follows by Theorem~\ref{thm:reachability_uniform}, that there exists some constant $C(\delta',A)$ such that uniformly in $\rho_{0}\in A$,
	\begin{equation}\label{eq:step_reachability}
		\mathbb{P}\bigl(\rds(t_{1},\place,\rho_{0})\in B^{-1}_{\delta'}(\unif)\bigr)\geq C(\delta',A)\,.
	\end{equation}

	\emph{Step 3 (regularisation):}
	An application of Lemma~\ref{lem:dynamics_regularisation} combined with~\eqref{eq:step_reachability} yields
	\begin{equation}\label{eq:step_regularisation}
		\mathbb{P}\bigl(\rds(t_{1},\place,\rho_{0})\in B_{\delta'}^{-1}(\unif),\,\rds(t_{1}+t_{2},\place,\rho_{0})\in B_{\delta}^{0}(\unif)\bigr)>\frac{1}{2}C(\delta',A)\,.
	\end{equation}
	
	\emph{Step 4 (decomposition):}
	An application of the cocycle property of $\rds$ yields
	\begin{equation*}
		\begin{split}
			&\mathbb{P}\bigl(\rds(t_{1}+t_{2}+n,\place,\rho_{0})\in B_{\eps}^{0}(\unif)\bigr)\\
			&\geq\PP\bigl(\rds(t_{1},\place,\rho_{0})\in B_{\delta'}^{-1}(\unif),\,\rds(t_{1}+t_{2},\place,\rho_{0})\in B_{\delta}^{0}(\unif),\\
			&\multiquad[3]\rds(t_{1}+t_{2}+n,\place,\rho_{0})\in B_{\eps}^{0}(\unif)\bigr)\\
			&=\mathbb{P}\bigl(\rds(t_{1},\place,\rho_{0})\in B_{\delta'}^{-1}(\unif),\,\rds(t_{1}+t_{2},\place,\rho_{0})\in B_{\delta}^{0}(\unif),\\
			&\multiquad[3]\rds(n,\shift_{t_{1}+t_{2}}\place,\rds(t_{1}+t_{2},\place,\rho_{0}))\in B_{\eps}^{0}(\unif)\bigr)\\
			&\geq\mathbb{P}\bigl(\rds(t_{1},\place,\rho_{0})\in B_{\delta'}^{-1}(\unif),\,\rds(t_{1}+t_{2},\place,\rho_{0})\in B_{\delta}^{0}(\unif),\\
			&\multiquad[3]\rds(n,\shift_{t_{1}+t_{2}}\place,B_{\delta}^{0}(\unif))\subset B_{\eps}^{0}(\unif)\bigr)\,.
		\end{split}
	\end{equation*}
	Note that $\rds(t_{1},\place,\rho_{0})$ and $\rds(t_{1}+t_{2},\place,\rho_{0})$ are $\mathcal{F}_{0,t_{1}+t_{2}}$-measurable, and that $\rds(n,\shift_{t_{1}+t_{2}}\place,B_{\delta}^{0}(\unif))$ is $\mathcal{F}_{t_{1}+t_{2},t_{1}+t_{2}+n}$-measurable. 
	Hence, it follows by the independence of $\mathcal{F}_{0,t_{1}+t_{2}}$ and $\mathcal{F}_{t_{1}+t_{2},t_{1}+t_{2}+n}$, that
	\begin{equation*}
		\begin{split}
			&\mathbb{P}\bigl(\rds(t_{1},\place,\rho_{0})\in B_{\delta'}^{-1}(\unif),\,\rds(t_{1}+t_{2},\place,\rho_{0})\in B_{\delta}^{0}(\unif),\\
			&\multiquad[3]\rds(n,\shift_{t_{1}+t_{2}}\place,B_{\delta}^{0}(\unif))\subset B_{\eps}^{0}(\unif)\bigr)\\
			&=\mathbb{P}\bigl(\rds(t_{1},\place,\rho_{0})\in B_{\delta'}^{-1}(\unif),\,\rds(t_{1}+t_{2},\place,\rho_{0})\in B_{\delta}^{0}(\unif)\bigr)\\
			&\quad\times\PP\bigl(\rds(n,\shift_{t_{1}+t_{2}}\place,B_{\delta}^{0}(\unif))\subset B_{\eps}^{0}(\unif)\bigr)\,.
		\end{split}
	\end{equation*}
	Using that $\PP$ is $(\shift_{t})_{t\in\RR}$-invariant, we obtain
	\begin{equation*}
		\PP\bigl(\rds(n,\shift_{t_{1}+t_{2}}\place,B_{\delta}^{0}(\unif))\subset B_{\eps}^{0}(\unif)\bigr)=\PP\bigl(\rds(n,\place,B_{\delta}^{0}(\unif))\subset B_{\eps}^{0}(\unif)\bigr)\,.
	\end{equation*}
	All in all,
	\begin{equation}\label{eq:step_decomposition}
		\begin{split}
			&\mathbb{P}\bigl(\rds(t_{1}+t_{2}+n,\place,\rho_{0})\in B_{\eps}^{0}(\unif)\bigr)\\
			&\geq\mathbb{P}\bigl(\rds(t_{1},\place,\rho_{0})\in B_{\delta'}^{-1}(\unif),\,\rds(t_{1}+t_{2},\place,\rho_{0})\in B_{\delta}^{0}(\unif)\bigr)\\
			&\quad\times\PP\bigl(\rds(n,\place,B_{\delta}^{0}(\unif))\subset B_{\eps}^{0}(\unif)\bigr)\,.
		\end{split}
	\end{equation}

	\emph{Step 5 (conclusion):}
	Combining~\eqref{eq:step_asymptotic_stability},~\eqref{eq:step_regularisation}~\&~\eqref{eq:step_decomposition}, we obtain that for all $\eps>0$ there exists some $N=N(\eps)\in\mathbb{N}$ such that uniformly in $n\geq N$ and $\rho_{0}\in A$,
	\begin{equation*}
		\mathbb{P}\bigl(\rds(t_{1}+t_{2}+n,\place,\rho_{0})\in B_{\eps}^{0}(\unif)\bigr)>\frac{1}{4}C(\delta',A)C(\delta)\,.
	\end{equation*}
	Therefore, it follows that uniformly in $\eps>0$ and $\rho_{0}\in A$,
	\begin{equation*}
		\liminf_{t\to\infty}\mathbb{P}\bigl(\rds(t,\place,\rho_{0})\in B_{\eps}^{0}(\unif)\bigr)\geq\frac{1}{4}C(\delta',A)C(\delta)\,,
	\end{equation*}
	which yields the claim.
\end{proof}
We are now in a position to prove our main result.
\begin{proof}[{Proof of Theorem~\ref{thm:main_result}}]
	Let $(t,\omega,\rho_{0})\mapsto\rds(t,\omega,\rho_{0})$ denote the RDS solution of the SPDE~\eqref{eq:SPDE} constructed in Theorem~\ref{thm:RDS_for_SPDE} and let $\meas\from\mathcal{B}(\init{0})\to[0,1]$
	be an invariant Borel probability measure such
	that $\meas\neq\delta_{\unif}$.
	It follows by the invariance of~$\meas$, that for all $t>0$ and $A\in\mathcal{B}(\init{0})$,
	\begin{equation*}
		\meas(A)=\int_{\init{0}}\PP\bigl(\rds(t,\place,\rho_{0})\in A\bigr)\,\meas(\dd\rho_{0})\,.
	\end{equation*}
	The regularity of solutions to~\eqref{eq:SPDE}, yields that~$\meas$ is supported on $\init{1}$. Further, an application of the Lusin--Souslin theorem (see~\cite[Thm.~15.1]{kechris_95}) yields $\mathcal{B}(\init{1})\subset\mathcal{B}(\init{0})$, which implies that~$\meas$ can be restricted to a Borel probability measure on $\init{1}$.
	
	The space $\init{1}$ is Polish, hence Radon, which implies
	that the Borel probability measure~$\meas$ is inner regular. We can then apply
	Lemma~\ref{lem:support_away} to deduce that there exists some compact $A\subset \init{1}$
	such that $\meas(A)>0$ and $\unif\not\in A$.

	Let $\eps>0$ and recall from Section~\ref{sec:prelim} that $B_{\eps}^{0}(\unif)$ denotes the open $L^{2}(\TT^{d})$-ball of radius~$\eps$ around~$\unif$. It follows by the invariance of~$\meas$,
	\begin{equation*}
		\begin{split}
			\meas(B_{\eps}^{0}(\unif))&=\int_{\init{1}}\mathbb{P}\bigl(\rds(t,\place,\rho_{0})\in B_{\eps}^{0}(\unif)\bigr)\,\meas(\dd\rho_{0})\\
			&=\meas(\{\unif\})+\int_{\init{1}\setminus\{\unif\}}\mathbb{P}\bigl(\rds(t,\place,\rho_{0})\in B_{\eps}^{0}(\unif)\bigr)\,\meas(\dd\rho_{0})\\
			&\geq\meas(\{\unif\})+\int_{A}\mathbb{P}\bigl(\rds(t,\place,\rho_{0})\in B_{\eps}^{0}(\unif)\bigr)\,\meas(\dd\rho_{0})\,,
		\end{split}
	\end{equation*}
	so that by Fatou's lemma,
	\begin{equation}\label{eq:mass_transport}
		\meas(B_{\eps}^{0}(\unif))\geq\meas(\{\unif\})+\int_{A}\liminf_{t\to\infty}\mathbb{P}\bigl(\rds(t,\place,\rho_{0})\in B_{\eps}^{0}(\unif)\bigr)\,\meas(\dd\rho_{0})\,.
	\end{equation}
	An application of Lemma~\ref{lem:key_reachability} yields that there exists some $c>0$ independent of~$\eps$ and~$\rho_{0}\in A$, such that,
	\begin{equation}\label{eq:key_reachability}
		\liminf_{t\to\infty}\mathbb{P}\bigl(\rds(t,\place,\rho_{0})\in B_{\eps}^{0}(\unif)\bigr)\geq c\,.
	\end{equation}
	Combining~\eqref{eq:mass_transport} and~\eqref{eq:key_reachability}, we obtain
	\begin{equation*}
		\meas(B_{\eps}^{0}(\unif))\geq\meas(\{\unif\})+c\meas(A)\,.
	\end{equation*}
	Letting $\eps\to0$, we obtain
	\begin{equation*}
		\meas(\{\unif\})\geq\meas(\{\unif\})+c\meas(A)\,,
	\end{equation*}
	which yields a contradiction, since $c>0$ and $\meas(A)>0$.
\end{proof}

\section{Explicit examples}
\label{sec:examples}
We now provide some explicit examples of interaction potentials for which the deterministic dynamics is not globally stable, but our main result (Theorem~\ref{thm:main_result}) applies to yield uniqueness of the invariant probability measure of the stochastic dynamics.

\subsection{Single-mode interaction potential: Second-order phase transition}
\label{subsec:single_mode}
We choose the interaction potential~$W$ to be given by
\begin{equation}
	W(x)=-N_k\prod_{i=1}^{d}\cos(2\pi k_{i}x_{i})\, ,
	\label{eq:single_mode_interaction}
\end{equation} 
where $k =(k_1,\ldots,k_d)\in\ZZ^d_0$ and $N_k$ is a normalisation constant such that $\norm{W}_{L^2(\TT^d)}=1$. This choice of interaction corresponds to a higher-dimensional version of the noisy Kuramoto model.  It is known (see~\cite{carrillo_gvalani_pavliotis_schlichting_20}) that this system exhibits a second-order (continuous) phase transition at the critical parameter value $\diffu_{\crit}=N_k^{-1}$. This means that for $\diffu>\diffu_{\crit}$, the uniform state $\unif$ is the unique minimiser of the free-energy functional $E$~\eqref{eq:free_energy_functional}, while for $\diffu<\diffu_{\crit}$, the minimisers are non-uniform steady states of the deterministic dynamics~\eqref{eq:McKean--Vlasov}.  At the critical value $\diffu=\diffu_{\crit}$, the uniform state $\unif$ loses its linear stability (i.e. the point at which eigenvalues of the operator $L$ cross the imaginary axis) and the free energy degenerates near $\unif$.  Thus new minimisers arise continuously from $\unif$ as $\diffu$ crosses below $\diffu_{\crit}$ (see Figure~\ref{fig:second_order_phase_transition}). We have the following corollary as a direct implication of the results of Theorem~\ref{thm:main_result}, see also Remark~\ref{rem:recovering_diffusivity}.
\begin{corollary}
	Let $W$ be as in~\eqref{eq:single_mode_interaction} and $\colour\in h^{\reg}(\ZZ_0^d)$ be non-trivial and radially symmetric with $\reg>4$. If $\diffu\geq N_k^{-1}$, then for any $\str>0$, the SPDE~\eqref{eq:SPDE_intro} admits the unique invariant probability measure $\delta_{\unif}$.

	On the other hand, if $\diffu< N_k^{-1}$ and 
	\begin{equation*}
	\abs{k}^2(N_k^{-1}-\diffu')-(\diffu-\diffu')<\norm{\colour}_{h^{-1}(\ZZ^{d}_{0})}^{2}C_{d}\str^{2}\, ,
	\end{equation*}
	for some $\diffu'<\diffu$, where $C^{(\diffu')}_{W}$ (resp.~$C_{d}$) is as in~\eqref{eq:maximal_growth_rate} (resp.~\eqref{eq:dimension_constant}), then the SPDE~\eqref{eq:SPDE_intro} admits the unique invariant probability measure~$\delta_{\unif}$.
\end{corollary}
The above result implies that when $\diffu\geq  N_k^{-1}$, the addition of even a small amount of noise ($\str>0$) is sufficient to ensure that~$\delta_{\unif}$ is the unique invariant probability measure of the stochastic dynamics.  More interestingly, when $\diffu<N_k^{-1}$, i.e.\ in the regime where the deterministic dynamics admits multiple stable steady states, then provided the noise strength $\str>0$ is sufficiently large (depending on how far $\diffu$ is below the critical value $N_k^{-1}$), we again have that~$\delta_{\unif}$ is the unique invariant probability measure of the stochastic dynamics. 
\begin{figure}[htb]
    \centering
	\begin{tikzpicture}[scale=1.2]
		%
		\draw [black,thick] plot [smooth, tension=1.3] coordinates {(-0.7,1.5) (0,0) (0.7,1.5)};
		\draw[blue,fill=blue] (0,0) circle (.5ex);

		\draw [black,thick] plot [smooth, tension=0.8] coordinates {(1.8,1.5) (2.0,0.39) (2.5,0) (3.0,0.39) (3.2,1.5)};
		\draw[blue,fill=blue] (2.5,0) circle (.5ex);

		\draw [black,thick] plot [smooth, tension=1.0] coordinates {(4.2,1.5) (4.4,-0.2) (5.0,0) (5.6,-0.2) (5.8,1.5)};
		\draw (5.0,0) node[shape=cross out, draw=red, line width=2pt, inner sep=0pt, minimum size=4pt] {};
		\draw[blue,fill=blue] (4.55,-0.35) circle (.5ex);
		\draw[blue,fill=blue] (5.45,-0.35) circle (.5ex);

		\draw[->,thick] (-0.7,-1) -- (5.8,-1);
		\node[above] at (0,-1) {\footnotesize $\diffu>\diffu_{\crit}$};
		\node[above] at (2.5,-1) {\footnotesize $\diffu=\diffu_{\crit}$};
		\node[above] at (5,-1) {\footnotesize $\diffu<\diffu_{\crit}$};
	\end{tikzpicture}
	\caption{Schematic depiction of the free-energy landscape for the single-mode interaction potential in Section~\ref{subsec:single_mode} as it undergoes a second-order phase transition with critical parameter $\diffu_{\crit}$. The blue dots indicate the stable steady states, while the red cross indicates an unstable one. The central point in each panel corresponds to the uniform state $\unif$.}
	\label{fig:second_order_phase_transition}
\end{figure}
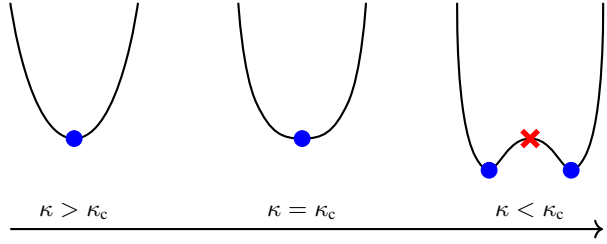
\begin{figure}[htb]
    \centering
    \begin{tikzpicture}
		%
		\draw [black,thick] plot [smooth, tension=1.5] coordinates {(-1,2) (0,0) (1,2)};
		\draw[blue,fill=blue] (0,0) circle (.5ex);

		\draw [black,thick] plot [smooth, tension=1.1] coordinates {(2,2) (2.25,0.7) (2.7,0.6) (3,0) (3.3,0.6) (3.75,0.7) (4,2)};
		\draw[blue,fill=blue] (3,0) circle (.5ex);
		\draw[blue,fill=blue] (3.6,0.52) circle (.5ex);
		\draw (3.4,0.62) node[shape=cross out, draw=red, line width=2pt, inner sep=0pt, minimum size=4pt] {};
		\draw[blue,fill=blue] (2.4,0.52) circle (.5ex);
		\draw (2.6,0.62) node[shape=cross out, draw=red, line width=2pt, inner sep=0pt, minimum size=4pt] {};

		\draw [black,thick] plot [smooth, tension=0.9] coordinates {(5,2) (5.25,0.7) (5.5,0) (5.7,0.62) (6,0) (6.3,0.62) (6.5,0) (6.75,0.7) (7,2)};
		\draw[blue,fill=blue] (6,0) circle (.5ex);
		\draw[blue,fill=blue] (5.5,0) circle (.5ex);
		\draw[blue,fill=blue] (6.5,0) circle (.5ex);
		\draw (5.7,0.62) node[shape=cross out, draw=red, line width=2pt, inner sep=0pt, minimum size=4pt] {};
		\draw (6.3,0.62) node[shape=cross out, draw=red, line width=2pt, inner sep=0pt, minimum size=4pt] {};

		\draw [black,thick] plot [smooth, tension=0.9] coordinates {(8,2) (8.25,0.7) (8.5,-0.3) (8.7,0.62) (9,0) (9.3,0.62) (9.5,-0.3) (9.75,0.7) (10,2)};
		\draw[blue,fill=blue] (9,0) circle (.5ex);
		\draw[blue,fill=blue] (8.5,-0.3) circle (.5ex);
		\draw[blue,fill=blue] (9.5,-0.3) circle (.5ex);
		\draw (8.7,0.62) node[shape=cross out, draw=red, line width=2pt, inner sep=0pt, minimum size=4pt] {};
		\draw (9.3,0.62) node[shape=cross out, draw=red, line width=2pt, inner sep=0pt, minimum size=4pt] {};

		\draw[->,thick] (-1,-1) -- (10,-1);
		\node[above] at (0,-1) {\footnotesize $\diffu\gg\diffu_{\crit}$};
		\node[above] at (3,-1) {\footnotesize $\diffu>\diffu_{\crit}$};
		\node[above] at (6,-1) {\footnotesize $\diffu=\diffu_{\crit}$};
		\node[above] at (9,-1) {\footnotesize $\diffu<\diffu_{\crit}$};
	\end{tikzpicture}
	\caption{Schematic depiction of the free-energy landscape for the two-mode interaction potential in Section~\ref{subsec:two_mode} as it undergoes a first-order phase transition with critical parameter $\diffu_{\crit}$. The blue dots indicate stable steady states, while the red crosses indicate unstable ones. The central point in each panel corresponds to the uniform state $\unif$.}
	\label{fig:first_order_phase_transition}
\end{figure}
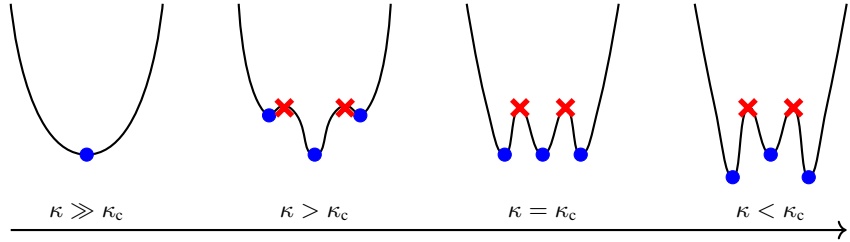
\subsection{Two-mode interaction potential: First-order phase transition}
\label{subsec:two_mode}
We now consider the potential
\begin{equation}
	W(x)=-N_{k}\prod_{i=1}^{d}\cos(2\pi k_{i}x_{i})-N_\ell \prod_{i=1}^{d}\cos(2\pi \ell_{i}x_{i})\, ,
	\label{eq:two_mode_interaction}
\end{equation}
where $k,\ell\in\ZZ^d_0$ are such that $k=2\ell$ and $N_k,N_\ell$ are normalisation constants such that $\norm{W}_{L^2(\TT^d)}=1$. Note that $k=2\ell$ enforces $N_k=N_\ell$. 
This choice of interaction potential is known to lead to a first-order (discontinuous) phase transition at some critical parameter value~$\diffu_{\crit}$ (see~\cite[Theorem 5.11]{carrillo_gvalani_pavliotis_schlichting_20}), i.e.\ there exists some~$\diffu_{\crit}$ such that for $\diffu>\diffu_{\crit}$, the uniform state~$\unif$ is the unique minimiser of the free-energy functional~$E$~\eqref{eq:free_energy_functional}, while for $\diffu<\diffu_{\crit}$, there exist additional minimisers which are non-uniform steady states of the deterministic dynamics~\eqref{eq:McKean--Vlasov}. In this setting, however the new minimisers arise discontinuously from~$\unif$ as~$\diffu$ crosses below~$\diffu_{\crit}$. This could happen, for example, through a saddle-node bifurcation where a pair of stable and unstable non-uniform steady states are created away from~$\unif$ (see Figure~\ref{fig:first_order_phase_transition}). The uniform state~$\unif$ remains linearly stable at and beyond the critical value $\diffu=\diffu_{\crit}$, i.e.\ eigenvalues of the operator~$L$ do not cross the imaginary axis at $\diffu=\diffu_{\crit}$. It can be shown that the uniform state $\unif$ loses its linear stability only at some $\diffu_{\sharp}<\diffu_{\crit}$, which is given exactly by $\diffu_{\sharp}=N_k^{-1}$. We have the following corollary as a direct implication of the results of Theorem~\ref{thm:main_result}.
\begin{corollary}
	Let $W$ be as in~\eqref{eq:two_mode_interaction} and $\colour\in h^{\reg}(\ZZ_0^d)$ be non-trivial and radially symmetric with $\reg>3$. If $\diffu\geq \diffu_{\sharp}=N_k^{-1}$, then for any $\str>0$, the SPDE~\eqref{eq:SPDE_intro} admits the unique invariant probability measure~$\delta_{\unif}$. 
	On the other hand, if $\diffu< \diffu_{\sharp}$ and
	\begin{equation*}
		\abs{k}^2(N_k^{-1}-\diffu')-(\diffu-\diffu')<\norm{\colour}_{h^{-1}(\ZZ^{d}_{0})}^{2}C_{d}\str^2\, ,
	\end{equation*}
	for some $\diffu'<\diffu$, where $C^{(\diffu')}_{W}$ (resp.~$C_{d}$) is as in~\eqref{eq:maximal_growth_rate} (resp.~\eqref{eq:dimension_constant}), then the SPDE~\eqref{eq:SPDE_intro} admits the unique invariant probability measure~$\delta_{\unif}$.
\end{corollary}

\paragraph{Acknowledgements.} 
BG acknowledges support by the Max Planck Society through the Research Group ``Stochastic Analysis in the Sciences.'' 
This work was funded by the European Union (ERC, FluCo, grant agreement No.\ 101088488). Views and opinions expressed are however those of the author(s) only and do not necessarily reflect those of the European Union or of the European Research Council. Neither the European Union nor the granting authority can be held responsible for them.
The work of RG is partially funded by the Deutsche Forschungsgemeinschaft (DFG, German Research Foundation) - SPP 2410 Hyperbolic Balance Laws in Fluid Mechanics: Complexity, Scales, Randomness (CoScaRa).


\appendix

\section{Well-posedness of stochastic characteristics}\label{app:stochastic_characteristics}

\subsection{Existence as a stochastic flow and random dynamical system}\label{subsec:stochastic_characteristics_RDS_flow}
In this subsection we show that the SDE~\eqref{eq:stochastic_characteristics_SDE} induces a stochastic flow of diffeomorphisms and a random dynamical system.
\begin{lemma}\label{lem:RDS_for_stochastic_characteristics}
	Let $\colour\in h^{1}(\ZZ_{0}^{d})$
	be radially symmetric and $\str>0$.
	Then there exists a unique (up to indistinguishability) continuous
	RDS $\flow\from[0,\infty)\times\Omega\times\TT^{d}\to\TT^{d}$ over
	$(\Omega,\mathcal{F}^{0},\PP,(\shift_{t})_{t\in\RR},(\mathcal{F}_{s,t})_{s\leq t})$
	which solves the SDE~\eqref{eq:stochastic_characteristics_SDE}.
	In particular, $\flow$ is a stochastic flow of homeomorphisms. 
	Further, for any $\gamma\in(0,1)$, it holds that~$\flow$ is a $C^{0,\gamma}(\TT^{d};\TT^{d})$-valued
	$(\mathcal{F}_{t})_{t\geq0}$-martingale.
	Let $m\in\NN$ and $\delta\in(0,1]$.
	If, in addition, $\colour\in h^{m+\delta}(\ZZ_{0}^{d})$, then~$\flow$
	is a $C^{m}(\TT^{d};\TT^{d})$-RDS, a stochastic flow of volume-preserving $C^{m}(\TT^{d};\TT^{d})$-diffeomorphisms and a $C^{m,\gamma}(\TT^{d};\TT^{d})$-valued $(\mathcal{F}_{t})_{t\geq0}$-martingale for any $\gamma\in(0,\delta)$.
\end{lemma}
\begin{proof}
	The generation of an RDS of the prescribed regularity follows by the same arguments as in the proof of~\cite[Thm.~2.3.26]{arnold_98}. In contrast to~\cite[Thm.~2.3.26]{arnold_98},
	we need to assume less decay of~$\colour$, since the Stratonovich correction of the SDE~\eqref{eq:stochastic_characteristics_SDE} vanishes.
		
	To prove that~$\flow$ is volume preserving, it suffices to show $\det\D\flow(t,\omega,x)=1$
	for every $t\geq0$, $\omega\in\Omega$ and $x\in\TT^{d}$. The dynamics of $\det\D\flow(t,\omega,x)$ is given by
	Liouville's equation,
	\begin{equation*}
		\begin{split}
			&\det\D\flow(t,\place,x)\\
			&=\exp\Bigl(-\sqrt{2}\str\sum_{k\in\mathbb{Z}_{0}^{d}}\colour_{k}\sum_{j=1}^{d-1}\int_{0}^{t}\trace(\D(a_{k}^{(j)}e_{k})(\flow(t,\place,x)))\circ\dd B^{(j)}(s,k)\Bigr)\,.
		\end{split}
	\end{equation*}
	Using that $k\cdot a_{k}^{(j)}=0$, it follows that $\trace(\D(a_{k}^{(j)}e_{k}))(x)=\vdiv(a_{k}^{(j)}e_{k})(x)=0$,
	which yields for every $t\geq0$, $\omega\in\Omega$ and $x\in\TT^{d}$, that $\det\D\flow(t,\omega,x)=1.$
\end{proof}
\subsection{The Wong--Zakai theorem for the stochastic characteristics}
In this subsection we present a Wong--Zakai approximation
for the stochastic characteristics constructed in Lemma~\ref{lem:RDS_for_stochastic_characteristics}.

For every $m\in\NN$, $j=1,\ldots,d-1$ and $k\in\ZZ_{0}^{d}$, let $(B_{m}^{(j)}(t,k))_{t\geq0}$
be the dyadic piecewise-linear approximation of the Brownian motion
$(B^{(j)}(\place,k))_{t\geq0}$ with breakpoints $D_{m}\defeq\{0\}\cup2^{-m}\NN$. Let $(\colour_{k})_{k\in\ZZ^{d}_{0}}$ be radially symmetric and $n\in\NN$. We then define the noise
\begin{equation*}
	\xi_{m,n}^{\colour}(t,x)\defeq\sum_{\substack{k\in\mathbb{Z}_{0}^{d}\\\abs{k}\leq n}}\sum_{j=1}^{d-1}\colour_{k}a_{k}^{(j)}e_{k}(x)\frac{\dd}{\dd t}B_{m}^{(j)}(t,k)\,.
\end{equation*}
Let $\str>0$ and~$\flow_{m,n}$ be the solution to the Young ODE,
\begin{equation*}
	\dd\flow_{m,n}(t,x)=-\sqrt{2}\str\sum_{\substack{k\in\mathbb{Z}_{0}^{d}\\\abs{k}\leq n}}\sum_{j=1}^{d-1}\colour_{k}a_{k}^{(j)}e_{k}(\flow_{m,n}(t,x))\,\dd B_{m}^{(j)}(t,k)\,,
\end{equation*}
with inital condition $\flow_{m,n}(0,x)=x$.

We can now present the Wong--Zakai theorem for the stochastic characteristics.
\begin{theorem}\label{thm:Wong_Zakai_via_rough_paths}
	Let $T>0$, $r\in\NN\cup\{0\}$, $\reg>r+2$, $\colour\in h^{\reg}(\ZZ_{0}^{d})$ be radially symmetric and $\str>0$.
	Then there exists some sequence $(m_{n})_{n\in\NN}$ such that $m_{n}\in\NN$ for every $n\in\NN$ and $m_{n}\to\infty$ as $n\to\infty$, such that 	
	\begin{equation*}
		\flow_{m_{n},n}\to\flow\qquad\text{in }C([0,T];C^{r}(\TT^{d};\TT^{d}))\,,\qquad\PP\text{-a.s.\ as }n\to\infty\,.
	\end{equation*}
\end{theorem}
\begin{proof}
	The claim follows by the proof of~\cite[Thm.~2.1]{dereich_dimitroff_12}.
\end{proof}
From now on we denote for every $n\in\NN$,
\begin{equation}\label{eq:Wong_Zakai_approximations}
	\xi_{n}^{\colour}\defeq\xi_{m_{n},n}^{\colour}\quad\text{and}\quad\flow_{n}\defeq\flow_{m_{n},n}\,.
\end{equation}
\subsection{The support theorem for the stochastic characteristics}
In this subsection we present a support theorem
for the stochastic characteristics constructed in Lemma~\ref{lem:RDS_for_stochastic_characteristics}.
\begin{theorem}\label{thm:support_theorem_flow}
	Let $r\in\NN\cup\{0\}$, $\reg>r+2$, $\colour\in h^{\reg}(\ZZ_{0}^{d})$ be radially symmetric and $\str>0$.
	Then, it holds that,
	\begin{equation*}
		\supp(\flow)=\cl\{\flow_{h}:h\in\cm(\colour)\}\,,
	\end{equation*}
	where $\cl$ is taken in $C([0,T];C^{r}(\mathbb{T}^{d};\mathbb{T}^{d}))$,
	$\cm(\colour)$ denotes the Cameron--Martin space defined by~\eqref{eq:Cameron--Martin_space}
	and for every $h\in\cm(\colour)$, $(t,x)\mapsto\flow_{h}(t,x)$ denotes the solution to the ODE
	\begin{equation}\label{eq:flow_skeleton}
		\dot{x}_{t}=-\sqrt{2}\str h_{t}(x_{t})\,,\qquad x_{0}=x\,.
	\end{equation}
\end{theorem}
\begin{proof}
	The claim follows by~\cite[Thm.~2.1]{dereich_dimitroff_12}.
\end{proof}
\section{Transport equations}\label{app:transport_equations}
Let $\str>0$, $u_{0}\in L^{2}_{0}(\TT^{d})$ and $u$ be a solution to the inviscid stochastic transport equation,
\begin{equation}\label{eq:stochastic_transport}
	\partial_{t}u=\sqrt{2}\str\nabla\cdot(u\circ\xi^{\colour})\,,\qquad u\tzero=u_{0}\,,
\end{equation}
which is pathwise unique if we assume that~$\colour$ is an element of $\ell^{2}(\ZZ^{d}_{0})$ and radially symmetric, see~\cite[Thm.~3]{galeati_20}. If further $\colour\in h^{\reg}(\ZZ^{d}_{0})$ for some $\reg>1$, then the solution can be represented
via the stochastic characteristics~$\flow$ constructed in Lemma~\ref{lem:RDS_for_stochastic_characteristics}, as
\begin{equation*}
	u(t,x;\omega,u_{0})=u_{0}(\flow^{-1}(t,\omega,x))\,,\qquad t\geq0\,,\quad x\in\TT^{d}\,,\quad\omega\in\Omega\,.
\end{equation*}
\subsection{Almost-sure mixing estimate}
In this subsection we present an almost-sure mixing estimate
for solutions to the inviscid stochastic transport equation~\eqref{eq:stochastic_transport}. Recall the defintion of the constant $C_{d}$ given in~\eqref{eq:dimension_constant}.
\begin{theorem}\label{thm:mixing_almost_surely_energy_spectrum_inviscid_transport}
	Let $\colour\in h^{1}(\ZZ_{0}^{d})$ be non-trivial and radially symmetric, and $\str>0$. Then for all $\gamma>0$ such that
	\begin{equation*}
		0<\gamma<\norm{\colour}_{h^{-1}(\ZZ^{d}_{0})}^{2}C_{d}\str^{2}\,,
	\end{equation*}
	there exists a random constant $D=D(\omega)$ depending on $d$, $\gamma$, $\str$ and $\norm{\colour}_{h^{1}(\ZZ^{d}_{0})}$, such that for all $u_{0}\in L^{2}_{0}(\TT^{d})$, $\PP$-a.s.\ in $\omega\in\Omega$,
	\begin{equation*}
		\lVert u(t,\place;\omega,u_{0})\rVert_{\dot{H}^{-1}(\TT^{d})}^{2}\leq D(\omega)\euler^{-2(2\uppi)^{2}\gamma t}\lVert u_{0}\rVert_{L^{2}(\TT^{d})}^{2}\,,\qquad\text{for all }t\geq0\,.
	\end{equation*}
\end{theorem}
\begin{proof}
	After rescaling, the result is given by~\cite[Thm.~1.2]{luo_tang_zhao_24}.
\end{proof}
\subsection{The Wong--Zakai theorem for stochastic transport equations}
In this subsection we present a Wong--Zakai type approximation result
for solutions to the inviscid stochastic transport equation~\eqref{eq:stochastic_transport}.
\begin{theorem}\label{thm:Wong_Zakai_stochastic_transport_equation}
	Let $T>0$, $\reg>3$, $\colour\in h^{\reg}(\ZZ^{d}_{0})$ be radially symmetric, $\str>0$ and $A\subset\dot{H}^{1}(\TT^{d})$ be bounded. For every $u_{0}\in A$ and $n\in\NN$, let $(\omega,t,x)\mapsto u_{n}(t,x;\omega,u_{0})$ be the solution to the associated transport equation
	\begin{equation*}
		\partial_{t}u_{n}=\sqrt{2}\str\nabla\cdot(u_{n}\xi_{n}^{\colour})\,,\qquad u_{n}\tzero=u_{0}\,,
	\end{equation*}
	where the noise $\xi_{n}^{\colour}$ is given by~\eqref{eq:Wong_Zakai_approximations}.
	Then it follows that $\PP$-a.s.\ in $\omega\in\Omega$,
	\begin{equation}\label{eq:Wong_Zakai_application}
		\lim_{n\to\infty}\sup_{u_{0}\in A}\norm{u(\place,\place;\omega,u_{0})-u_{n}(\place,\place;\omega,u_{0})}_{C([0,T];L^{2}(\TT^{d}))}=0\,,
	\end{equation}
	where $(\omega,t,x)\mapsto u(t,x;\omega,u_{0})$ is the solution to~\eqref{eq:stochastic_transport}.
\end{theorem}
\begin{proof}
	For every $n\in\NN$ define the stochastic characteristics~$\flow_{n}$ as in~\eqref{eq:Wong_Zakai_approximations} and note that $u_{n}(t,x;\omega,u_{0})=u_{0}(\flow_{n}^{-1}(t,\omega,x))$.
	It follows by the Wong--Zakai theorem (see Theorem~\ref{thm:Wong_Zakai_via_rough_paths}), that $\PP$-a.s.\ in $\omega\in\Omega$,
	\begin{equation*}
		\lim_{n\to\infty}\norm{\flow(\omega)-\flow_{n}(\omega)}_{C([0,T];C^{1}(\TT^{d};\TT^{d}))}=0\,.
	\end{equation*}
	Using that $A$ is bounded in $\dot{H}^{1}(\TT^{d})$, hence relatively compact in $L^{2}(\TT^{d})$, the claim follows by Lemma~\ref{lem:uniform_continuity_L2_concatenation_diffeomorphism_flow_inverse}.
\end{proof}
\subsection{Existence of mixers in the Cameron--Martin space}
In this subsection we  combine the almost-sure mixing estimate of Theorem~\ref{thm:mixing_almost_surely_energy_spectrum_inviscid_transport} and the Wong--Zakai approximation of Theorem~\ref{thm:Wong_Zakai_stochastic_transport_equation} to show that there exist mixing velocity fields in the Cameron--Martin space~\eqref{eq:Cameron--Martin_space}.
\begin{lemma}\label{lem:existence_Cameron_Martin_mixer}
	Let $\reg>3$, $\colour\in h^{\reg}(\ZZ^{d}_{0})$ be non-trivial and radially symmetric, $\str>0$,
	$\delta>0$ and $A\subset\dot{H}^{1}(\TT^{d})$ be bounded.
	Then there exists some
	$t_{1}=t_{1}(\delta,A,\str,\colour)\in[0,\infty)$ and $v=v(\delta,A,\str,\colour)\in\cm(\colour)$ such
	that for any initial data $u_{0}\in A$ and any solution $(t,x)\mapsto u_{v}(t,x;u_{0})$
	to 
	\begin{equation*}
		\partial_{t}u=\sqrt{2}\str\nabla\cdot(uv)\,,\qquad u\tzero=u_{0}\,,
	\end{equation*}
	it holds that uniformly in $u_{0}\in A$,
	\begin{equation*}
		\lVert u_{v}(t_{1},\place;u_{0})\rVert_{\dot{H}^{-1}(\T^{d})}\leq\delta\,.
	\end{equation*}
\end{lemma}

\begin{proof}
	Let $\reg>3$, $\colour\in h^{\reg}(\ZZ^{d}_{0})$ be non-trivial and radially symmetric, $\str>0$, $u_{0}\in\dot{H}^{1}(\TT^{d})$ and $(\omega,t,x)\mapsto u(t,x;\omega,u_{0})$ be the solution to the transport SPDE~\eqref{eq:stochastic_transport}.
	Recall the definition of $C_{d}$ (see~\eqref{eq:dimension_constant}.)
	It then follows by Lemma~\ref{thm:mixing_almost_surely_energy_spectrum_inviscid_transport},
	that for all
	\begin{equation*}
		0<\gamma<\norm{\colour}_{h^{-1}(\ZZ^{d}_{0})}^{2}C_{d}\str^{2}\,,
	\end{equation*}
	there exists a random constant $D=D(\omega)$ independent of $u_{0}$
	(and depending on $d$, $\gamma$, $\str$ and $\norm{\colour}_{h^{1}(\ZZ^{d}_{0})}$) such that for
	all $t\in[0,\infty)$,
	\begin{equation}\label{eq:mixing_application}
		\lVert u(t,\place;\omega,u_{0})\rVert_{\dot{H}^{-1}(\TT^{d})}^{2}\leq D(\omega)\euler^{-2(2\uppi)^{2}\gamma t}\lVert u_{0}\rVert_{L^{2}(\TT^{d})}^{2}\,,\qquad\mathbb{P}\text{-a.s.}
	\end{equation}
	Denote by $N_{1}\subset\Omega$ the null set outside of which~\eqref{eq:mixing_application}
	holds true. Further, denote by $N_{2}^{(T)}\subset\Omega$ the null set outside of which the Wong--Zakai approximation~\eqref{eq:Wong_Zakai_application}
	holds true. Define
	\begin{equation*}
		N_{2}\defeq\bigcup_{T\in\NN}N_{2}^{(T)}
	\end{equation*}
	and note that~$N_{2}$ is again of measure~$0$ as a countable union
	of null sets.

	For every $n\in\NN$, let~$\xi_{n}^{\colour}$ be the Wong--Zakai noise given by~\eqref{eq:Wong_Zakai_approximations}. Fix a realization $\omega_{*}\in\Omega\setminus\{N_{1}\cup N_{2}\}$
	and define the deterministic velocity field
	\begin{equation*}
		v_{n}(t,x)\defeq\xi_{n}^{\colour}(t,x;\omega_{*})\,.
	\end{equation*}
	We obtain by the triangle inequality for every $n\in\NN$ and $t\in[0,\infty)$,
	\begin{equation}\label{eq:triangle_inequality_application}
		\begin{split}
			&\norm{u_{v_{n}}(t,\place;u_{0})}_{\dot{H}^{-1}(\TT^{d})}\\
			&\qquad\leq\norm{u(t,\place;\omega_{*},u_{0})}_{\dot{H}^{-1}(\TT^{d})}\\
			&\qquad\quad+\norm{u(\place,\place;\omega_{*},u_{0})-u_{v_{n}}(\place,\place;u_{0})}_{C([0,t];\dot{H}^{-1}(\TT^{d}))}\,.
		\end{split}
	\end{equation}
	By the choice of $\omega_{*}$ and~\eqref{eq:mixing_application},
	it follows that for all $t\in[0,\infty)$,
	\begin{equation*}
		\norm{u(t,\place;\omega_{*},u_{0})}_{\dot{H}^{-1}(\TT^{d})}^{2}\leq D(\omega_{*})\euler^{-2(2\uppi)^{2}\gamma t}\lVert u_{0}\rVert_{L^{2}(\TT^{d})}^{2}\,.
	\end{equation*}
	Using that $A\subset\dot{H}^{1}(\TT^{d})\embed L^{2}_{0}(\TT^{d})$ is bounded, we can choose $t_{1}=t_{1}(\gamma,\delta,A,\omega_{*})$
	such that uniformly over $u_{0}\in A$, 
	\begin{equation}\label{eq:mixing_realization}
		\norm{u(t_{1},\place;\omega_{*},u_{0})}_{\dot{H}^{-1}(\TT^{d})}\leq\delta/2\,.
	\end{equation}
	By the choice of~$\omega_{*}$ and Theorem~\ref{thm:Wong_Zakai_stochastic_transport_equation}
	there exists some $n_{*}=n_{*}(t_{1},\delta,A,\omega_{*})$ such that uniformly in $u_{0}\in A$,
	\begin{equation}\label{eq:Wong_Zakai_realization}
		\begin{split}
			&\norm{u(\place,\place;\omega_{*},u_{0})-u_{v_{n_{*}}}(\place,\place;u_{0})}_{C([0,t_{1}];\dot{H}^{-1}(\TT^{d}))}\\
			&\leq\norm{u(\place,\place;\omega_{*},u_{0})-u_{v_{n_{*}}}(\place,\place;u_{0})}_{C([0,t_{1}];L^{2}(\TT^{d}))}\\
			&\leq\delta/2\,.
		\end{split}
	\end{equation}
	Combining~\eqref{eq:triangle_inequality_application},~\eqref{eq:mixing_realization}~\&~\eqref{eq:Wong_Zakai_realization},
	we obtain
	\begin{equation*}
		\norm{u_{v_{n_{*}}}(t_{1},\place;u_{0})}_{\dot{H}^{-1}(\TT^{d})}\leq\delta\,.
	\end{equation*}
	
	Setting $v(t,x)\defeq v_{n_{*}}(t,x)$ for every $t\in[0,t_{1}]$ and
	$x\in\TT^{d}$, it remains to show that $v\in\cm(\colour)$.
	Using that~$v$ has only finitely many Fourier modes, it suffices to show that $(\frac{\dd}{\dd t}B_{m}^{(j)})(\place,k;\omega_{*})\in L^{2}([0,t_{1}])$
	for every $m\in\NN$, $j=1,\ldots,d-1$ and $k\in\ZZ_{0}^{d}$. Each $B_{m}^{(j)}(\place,k;\omega_{*})$ is piecewise linear, hence
	Lipschitz continuous, which yields the claim.
\end{proof}
\subsection{The approximation error between viscous and inviscid transport}
\begin{lemma}\label{lem:zeroth_order_bound}
	Let $A\subset H^{1}(\TT^{d})$ be bounded and $v\in L^1([0,\infty);C^{1}(\TT^{d};\RR^{d}))$ be incompressible.
	Denote for every $\scale\geq0$ and $\rho_{0}\in A$ by~$(t,x)\mapsto\rho^{(\scale)}_{v}(t,x;\rho_{0})$
	the solution to the PDE
	\begin{equation*}
		\partial_{t}\rho=\scale\bigl(\Delta\rho+\nabla\cdot(\rho(\nabla W\ast\rho))\bigr)+\nabla\cdot(\rho v)\,,\qquad\rho\tzero=\rho_{0}\,.
	\end{equation*}
	Then, there exists some time-dependent, positive function
	\begin{equation*}
		C_{t}=C_{t}(A,\norm{v}_{L^{1}([0,t];C^{1}(\TT^{d}))},\norm{W}_{C^{2}(\TT^{d})})
	\end{equation*}
	such that for every $t\geq0$ and $\alpha>0$ uniformly over $\rho_{0}\in A$,
	\begin{equation*}
		\norm{\rho^{(\scale)}_{v}(t,\place;\rho_{0})-\rho^{(0)}_{v}(t,\place;\rho_{0})}_{L^{2}(\TT^{d})}\leq\scale^{1/2}C_{t}\,.
	\end{equation*}
\end{lemma}
\begin{proof}
	The claim follows by standard energy estimates.
\end{proof}
\section{Joint measurability of the SPDE}\label{app:joint_measurability}
Let $W\in C^{2}(\TT^{d})$, $\reg>2$, $\colour\in h^{\reg}(\ZZ_{0}^{d})$ be radially symmetric and $\str>0$. In this appendix we show that the solution map $P:[0,\infty)\times\Omega\times\init{0}\to\init{0}$ to the SPDE~\eqref{eq:SPDE} is $\bigl(\mathcal{B}([0,\infty))\otimes\mathcal{F}^{0}\otimes\mathcal{B}(\init{0}),\mathcal{B}(\init{0})\bigr)$-measurable.

We first prove the claim for finite time horizons.
\begin{lemma}\label{lem:joint_measurability_for_RDS_finite_time_horizon}
	Let $T>0$
	and $\rds\from[0,T]\times\Omega\times\init{0}\to\init{0}$
	be the solution map to the SPDE~\eqref{eq:SPDE}. Then, $\rds$ is
	$\bigl(\mathcal{B}([0,T])\otimes\mathcal{F}^{0}\otimes\mathcal{B}(\init{0}),\mathcal{B}(\init{0})\bigr)$-measurable.
\end{lemma}
\begin{proof}
	Denote by $\rho(t,x;\flow(\omega),\rho_{0})$
	the solution to the SPDE~\eqref{eq:SPDE} constructed in Lemma~\ref{lem:well_posedness_SPDE}. It follows that $\rds(t,\omega,\rho_{0})(x)=\rho(t,x;\flow(\omega),\rho_{0})$,
	where by Lemma~\ref{lem:continuity_flow_to_SPDE}, $\flow\mapsto\rho(\place,\place;\flow,\rho_{0})$ is continuous
	from $C([0,T];\diffvol{2})$ to $C([0,T];\init{0})$. 
	
	Assume for the moment that $\omega\to\flow(\omega)$ is $\bigl(\mathcal{F}^{0},\mathcal{B}(C([0,T];\diffvol{2}))\bigr)$-measurable as a map from $\Omega$ to $C([0,T];\diffvol{2})$. 
	It then follows that
	$\omega\mapsto\rho(t,\place;\flow(\omega),\rho_{0})$ is $\bigl(\mathcal{F}^{0},\mathcal{B}(\init{0})\bigr)$-measurable
	for every fixed $t\in[0,T]$. Further, $t\mapsto\rho(t,\place;\flow(\omega),\rho_{0})$
	is continuous from $[0,T]$ to $\init{0}$ for every fixed $\omega\in\Omega$.
	Hence, $(t,\omega)\mapsto\rho(t,\place;\flow(\omega),\rho_{0})$ is a Carathéodory
	function (cf.~\cite[Def.~4.50]{aliprantis_border_06}), and an
	application of~\cite[Lem.~4.51]{aliprantis_border_06} yields
	that it is $\bigl(\mathcal{B}([0,T])\otimes\mathcal{F}^{0},\mathcal{B}(\init{0})\bigr)$-measurable.

	Next, for fixed $(t,\omega)\in[0,T]\times\Omega$, it follows by Lemma~\ref{lem:well_posedness_SPDE} that the map $\rho_{0}\mapsto\rho(t,\place;\flow(\omega),\rho_{0})$ is continuous from $\init{0}$ to $\init{0}$. 
	Therefore,
	another application of~\cite[Lem.~4.51]{aliprantis_border_06}
	yields that the map $(t,\omega,\rho_{0})\mapsto\rds(t,\omega,\rho_{0})=\rho(t,\place;\flow(\omega),\rho_{0})$
	is $\bigl(\mathcal{B}([0,T])\otimes\mathcal{F}^{0}\otimes\mathcal{B}(\init{0}),\mathcal{B}(\init{0})\bigr)$-measurable.

	Hence, it suffices to show that $\omega\to\flow(\omega)$ is $\bigl(\mathcal{F}^{0},\mathcal{B}(C([0,T];\diffvol{2}))\bigr)$-measurable
	as a map from $\Omega$ to $C([0,T];\diffvol{2})$. Using that the embedding
	$\diffvol{2}\subset C(\TT^{d};\TT^{d})$ is continuous, it follows by
	the Lusin--Souslin theorem~\cite[Thm.~15.1]{kechris_95}, that
	$\mathcal{B}(\diffvol{2})\subset\mathcal{B}(C(\TT^{d};\TT^{d}))$. In
	particular, it is sufficient to show that $\flow$ is $\bigl(\mathcal{F}^{0},\mathcal{B}(C([0,T];C(\TT^{d};\TT^{d})))\bigr)$-measurable.
	Using that $C([0,T];C(\TT^{d};\TT^{d}))\simeq C([0,T]\times\TT^{d};\TT^{d})$,
	it suffices further to show that $\flow$ is $\bigl(\mathcal{F}^{0},\mathcal{B}(C([0,T]\times\TT^{d};\TT^{d}))\bigr)$-measurable.

	For fixed $(t,x)\in[0,T]\times\TT^{d}$, it follows by the joint measurability
	of the~RDS~$\flow$ that $\omega\mapsto\flow(t,\omega,x)$ is $(\mathcal{F}^{0},\mathcal{B}(\TT^{d}))$-measurable.
	On the other hand for fixed $\omega\in\Omega$, it follows by $\flow$
	being a continuous RDS,
	that the map $(t,x)\mapsto\flow(t,\omega,x)$ is continuous from $[0,T]\times\TT^{d}$
	to $\TT^{d}$. Hence, $\flow\from[0,T]\times\Omega\times\TT^{d}\to\TT^{d}$
	is a Carathéodory function. Using the compactness of $[0,T]\times\TT^{d}$,
	it follows by an application of~\cite[Thm.~4.55]{aliprantis_border_06}
	that $\flow$ is $\bigl(\mathcal{F}^{0},\mathcal{B}(C([0,T]\times\TT^{d};\TT^{d}))\bigr)$-measurable.
	This yields the claim.
\end{proof}
Next, let us prove joint measurability for infinite time horizons.
\begin{lemma}\label{lem:joint_measurability_for_RDS}
	Let $\rds\from[0,\infty)\times\Omega\times\init{0}\to \init{0}$
	be the solution map to the SPDE~\eqref{eq:SPDE}. Then, $\rds$ is
	$\bigl(\mathcal{B}([0,\infty))\otimes\mathcal{F}^{0}\otimes\mathcal{B}(\init{0}),\mathcal{B}(\init{0})\bigr)$-measurable.
\end{lemma}
\begin{proof}
	For every $T>0$, define
	\begin{equation*}
		P^{(T)}(t,\omega,\rho_{0})
		\defeq
		\begin{cases}
			P(t,\omega,\rho_{0}) & t\leq T\\
			P(T,\omega,\rho_{0}) & t>T
		\end{cases}
		\,.
	\end{equation*}
	In particular, $\lim_{T\to\infty}P^{(T)}(t,\omega,\rho_{0})=P(t,\omega,\rho_{0})$
	for every $(t,\omega,\rho_{0})\in[0,\infty)\times\Omega\times \init{0}$.
	An application of Lemma~\ref{lem:joint_measurability_for_RDS_finite_time_horizon}
	yields that each $P^{(T)}$ is $\bigl(\mathcal{B}([0,T])\otimes\mathcal{F}^{0}\otimes\mathcal{B}(\init{0}),\mathcal{B}(\init{0})\bigr)$-measurable,
	which, using the inclusion $\mathcal{B}([0,T])\subset\mathcal{B}([0,\infty))$,
	yields that each $P^{(T)}$ is $\bigl(\mathcal{B}([0,\infty))\otimes\mathcal{F}^{0}\otimes\mathcal{B}(\init{0}),\mathcal{B}(\init{0})\bigr)$-measurable.
	The claim follows since the pointwise limit of measurable functions
	is measurable~\cite[Lem.~4.29]{aliprantis_border_06}. 
\end{proof}
\section{Continuity of the inversion map}\label{app:continuity_inversion_map}
Recall that we denote for $\flow\in C([0,\infty);\homeo)$ the \emph{spatial} inverse by~$\flow^{-1}$. In this appendix we show that the inversion map is continuous when considered as a map between spaces of time-dependent diffeomorphisms (Lemma~\ref{lem:inversion_map_continuity_time_dependent}). We then use this to show continuity of the map $\flow\mapsto u\circ\flow^{-1}$ for fixed $u\in C([0,\infty);L^{2}(\TT^{d}))$ (Lemma~\ref{lem:continuity_CTL2_concatenation_diffeomorphism_flow_inverse}). By a similar argument we also show the uniform continuity of the map $(u,\flow)\mapsto u\circ\flow^{-1}$ if $u$ is taken from a compact subset of $L^{2}(\TT^{d})$ (Lemma~\ref{lem:uniform_continuity_L2_concatenation_diffeomorphism_flow_inverse}). For reasons of generality, we prove these results for arbitrary, not necessarily volume-preserving, diffeomorphisms.
\begin{lemma}\label{lem:inversion_map_continuity_time_dependent}
	Let $T>0$. The map $\flow\mapsto\flow^{-1}$ is continuous from $C([0,T];\diff{1})$ to $C([0,T];\homeo)$ and from $C([0,T];\diff{m+1})$ to $C([0,T];\diff{m})$ for every $m\in\NN$.
\end{lemma}
\begin{proof}
	For the sake of avoiding a case distinction, let us denote $\diff{0}\defeq\homeo$. Let $m\in\NN\cup\{0\}$, $\flow\in C([0,T];\diff{m+1})$
	and $(\flow_{n})_{n\in\NN}$ be such that $\flow_{n}\to\flow$ in
	$C([0,T];\diff{m+1})$ as $n\to\infty$. It suffices to show that
	$\flow_{n}^{-1}\to\flow^{-1}$ in $C([0,T];\diff{m})$ as $n\to\infty$.
	Using that $\flow_{n}\to\flow$ in $C([0,T];\diff{m+1})$ as $n\to\infty$,
	it follows that $K\defeq\{\flow_{n}(t):n\in\mathbb{N},\,t\in[0,T]\}\cup\{\flow(t):t\in[0,T]\}$
	is closed and bounded in $C^{m+1}(\mathbb{T}^{d};\mathbb{T}^{d})$,
	hence compact in $C^{m}(\mathbb{T}^{d};\mathbb{T}^{d})$. Further,
	the inversion map $I(\flow)\coloneqq\flow^{-1}$ is continuous from
	$\diff{m}$ to $\diff{m}$,
	which yields that $I$ is uniformly continuous on $K$. Consequently,
	it follows by the convergence $\lim_{n\to\infty}\lVert\flow-\flow_{n}\rVert_{C([0,T];C^{m}(\mathbb{T}^{d};\mathbb{T}^{d}))}=0$,
	that $\lim_{n\to\infty}\lVert\flow^{-1}-\flow_{n}^{-1}\rVert_{C([0,T];C^{m}(\mathbb{T}^{d};\mathbb{T}^{d}))}=0$.
\end{proof}
\begin{lemma}\label{lem:continuity_CTL2_concatenation_diffeomorphism_flow_inverse}
	Let $T>0$ and $u\in C([0,T];L^{2}(\mathbb{T}^{d}))$. Then, the map $\flow\mapsto u\circ\flow^{-1}$ is continuous from $C([0,T];\diff 1)$ to $C([0,T];L^{2}(\T^{d}))$.
\end{lemma}
\begin{proof}
	We mimic the proof of strong continuity of the translation semigroup, see e.g.~\cite[Chpt.~I, Subsec.~3.c]{engel_nagel_06}.
	Let $\flow\in C([0,T];\diff 1)$ and $(\flow_{n})_{n\in\mathbb{N}}$
	be a sequence such that $\flow_{n}\in C([0,T];\diff 1)$ for every
	$n\in\mathbb{N}$ and $\flow_{n}\to\flow$ in $C([0,T];\diff 1)$
	as $n\to\infty$. Using that $C([0,T];\diff 1)$ is a metric space,
	it suffices to show that $\lim_{n\to\infty}u\circ\flow_{n}^{-1}=u\circ\flow^{-1}$
	in $C([0,T];L^{2}(\T^{d}))$ for every $u\in C([0,T];L^{2}(\T^{d}))$.
	Combining the density of $C^{\infty}(\mathbb{T}^{d})$ in $L^{2}(\mathbb{T}^{d})$
	with a partition of unity (in time), it follows that $C([0,T];C^{\infty}(\T^{d}))$
	is dense in $C([0,T];L^{2}(\T^{d}))$. An application of~\cite[Chpt.~I, Lem.~1.2]{engel_nagel_06} then yields that it suffices to show that the sequence of
	linear operators $(F_{1/n})_{n\in\mathbb{N}}$ given by $(F_{1/n}u)(t,x)\defeq u(t,\flow_{n}^{-1}(t,x))$
	is uniformly bounded in $L(C([0,T];L^{2}(\mathbb{T}^{d})))$ and that
	$\lim_{n\to\infty}F_{1/n}u=F_{0}u\defeq u\circ\flow$ for all $u\in C([0,T];C^{\infty}(\T^{d}))$.

	We first show uniform boundedness. It follows by the integral transformation
	theorem, that for all $t\in[0,T]$,
	\begin{equation}\label{eq:concatenation_bound_inverse}
		\begin{split}
			\lVert(F_{1/n}u)(t)\rVert_{L^{2}(\mathbb{T}^{d})}^{2}&=\int_{\mathbb{T}^{d}}\lvert u(t,\flow_{n}^{-1}(t,x))\rvert^{2}\,\dd x\\
			&=\int_{\mathbb{T}^{d}}\lvert u(t,x)\rvert^{2}\lvert\det\D\flow_{n}(t,x)\rvert\,\dd x\\
			&\leq\lVert u(t,\place)\rVert_{L^{2}(\mathbb{T}^{d})}^{2}\lVert\det\D\flow_{n}(t,\place)\rVert_{L^{\infty}(\mathbb{T}^{d})}\,.
		\end{split}
	\end{equation}
	Using that $\flow_{n}\to\flow$ in $C([0,T];\diff{1})$ as $n\to\infty$,
	we can deduce 
	\begin{equation}\label{eq:Jacobi_matrix_uniform_control}
		\sup_{n\in\mathbb{N}}\sup_{t\in[0,T]}\lVert\det \D\flow_{n}(t,\place)\rVert_{L^{\infty}(\mathbb{T}^{d})}<\infty\,.
	\end{equation}
	Combining~\eqref{eq:concatenation_bound_inverse}~and~\eqref{eq:Jacobi_matrix_uniform_control},
	we obtain that $(F_{1/n})_{n\in\mathbb{N}}$ is bounded uniformly
	in $n\in\mathbb{N}$ as linear operators from $C([0,T];L^{2}(\mathbb{T}^{d}))$
	to $C([0,T];L^{2}(\mathbb{T}^{d}))$.

	Next we show that $\lim_{n\to\infty}F_{1/n}u=F_{0}u=u\circ\flow^{-1}$
	in $C([0,T];L^{2}(\mathbb{T}^{d}))$ for all $u\in C([0,T];C^{\infty}(\mathbb{T}^{d}))$.
	The convergence $\lim_{n\to\infty}\flow_{n}=\flow$ in $C([0,T];\diff 1)$,
	combined with an application of Lemma~\ref{lem:inversion_map_continuity_time_dependent}
	yields $\lim_{n\to\infty}\flow_{n}^{-1}=\flow^{-1}$ in $C([0,T];\homeo)$
	as $n\to\infty$. Since $[0,T]\times\mathbb{T}^{d}$ is compact and
	\begin{equation*}
		C([0,T];C(\mathbb{T}^{d}))\simeq C([0,T]\times\mathbb{T}^{d})\,,
	\end{equation*}
	it follows that $u$ is uniformly continuous in space-time. Therefore, we can conclude
	that $\lim_{n\to\infty}\lVert F_{0}u-F_{1/n}u\rVert_{C([0,T];C(\mathbb{T}^{d}))}=0$.
	Using the embedding $C(\mathbb{T}^{d})\hookrightarrow L^{2}(\mathbb{T}^{d})$,
	this yields the claim.
\end{proof}
\begin{lemma}\label{lem:uniform_continuity_L2_concatenation_diffeomorphism_flow_inverse}
	Let $T>0$, $A\subset L^{2}(\TT^{d})$ be relatively compact and $(\flow_{n})_{n\in\NN}$ be a sequence such that $\flow_{n}\in C([0,T];\diff{1})$ for every $n\in\NN$ and $\flow_{n}\to\flow$ in $C([0,T];\diff{1})$ as $n\to\infty$.
	Then, it holds that 
	\begin{equation*}
		\lim_{n\to\infty}\sup_{u_{0}\in\cl A}\norm{u_{0}\circ\flow_{n}^{-1}-u_{0}\circ\flow^{-1}}_{C([0,T];L^{2}(\TT^{d}))}=0\,.
	\end{equation*}
\end{lemma}
\begin{proof}
	Define the sequence of
	linear operators $(F_{1/n})_{n\in\mathbb{N}}$ acting on $u_{0}\in L^{2}(\TT^{d})$ by $(F_{1/n}u_{0})(t,x)\coloneqq u_{0}(\flow_{n}^{-1}(t,x))$. As in the proof of Lemma~\ref{lem:continuity_CTL2_concatenation_diffeomorphism_flow_inverse},
	it follows that $(F_{1/n})_{n\in\mathbb{N}}$ is uniformly bounded in $L(L^{2}(\mathbb{T}^{d});C([0,T];L^{2}(\mathbb{T}^{d})))$
	and that $\lim_{n\to\infty}F_{1/n}u_{0}=u_{0}\circ\flow^{-1}\defeq F_{0}u_{0}$
	for all $u_{0}\in C^{\infty}(\mathbb{T}^{d})$. An application of~\cite[Chpt.~I, Lem.~1.2]{engel_nagel_06} yields that the map $(1/n,u_{0})\mapsto F_{1/n}u_{0}$ is uniformly continuous on $\mathbb{N}\times\cl A$, which yields the claim.
\end{proof}
\section{Supports of measures}\label{app:supports_of_measures}
In this appendix we collect two auxiliary results on supports of measures that are used in the main text. Lemma~\ref{lem:support_push_forward} shows how the support of a measure transforms under pushforward by a continuous map. Lemma~\ref{lem:support_away} shows that for every non-trivial, inner-regular, Borel probability measure on a metric space, one can always find a non-trivial compact set of positive probability.
\begin{lemma}\label{lem:support_push_forward}
	Let $X$, $Y$ be separable metric spaces, $f\from X\to Y$ be continuous and $\mu$ be a Borel probability measure on $X$. Then, it follows that
	\begin{equation*}
		\supp(f_{\#}\mu)=\cl f(\supp(\mu))\,.
	\end{equation*}
\end{lemma}
\begin{proof}
	See e.g.~\cite[(5.2.6)]{ambrosio_gigli_savare_08}.
\end{proof}
\begin{lemma}\label{lem:support_away}
	Let $X$ be a metric space and $\meas$ be
	an inner-regular (i.e.~tight) Borel probability measure on $X$. Let $x\in X$ and assume $\meas\neq\delta_{x}$.
	Then, there exists a compact subset $K$ such that $\meas(K)>0$ and
	$x\notin K$.
\end{lemma}
\begin{proof}
	Since $\meas\neq\delta_{x}$, there exists some $A\in\mathcal{B}(X)$
	such that $\meas(A)\neq\delta_{x}(A)$. Without limitation of generality,
	we may assume $A$ to be closed. We now distinguish the cases $x\in A$
	and $x\notin A$. 

	If $x\in A$, then $1=\delta_{x}(A)\neq\meas(A)$, so that $\meas(A)<1$.
	In particular, $\meas(A^{\complement})>0$. Since $A^{\complement}$ is open, it follows
	by the inner regularity of $\meas$, that
	\begin{equation*}
		0<\meas(A^{\complement})=\sup_{\substack{K\subset A^{\complement}\\K\,\text{compact}}}\meas(K)\,.
	\end{equation*}
	In particular, we can find some compact set $K$ such that $\meas(K)>0$
	and $x\notin K$.

	If $x\notin A$, then $0=\delta_{x}(A)\neq\meas(A)$, so that $\meas(A)>0$.
	Since $X$ is a metric space, it is completely regular (i.e.~Tychonoff)
	and, therefore, regular. Consequently, there exists some open neighbourhood
	$U$ of $A$ such that $x\notin U$. In particular, $\meas(U)\geq\meas(A)>0$.
	We can then argue as in the previous case to find a compact $K$ such
	that $\meas(K)>0$ and $x\notin K$. This yields the claim.
\end{proof}

\endappendix

\small
\bibliographystyle{Martin}
\bibliography{bibliography/bibliography.bib}
\end{document}